\documentclass{article}

\usepackage{amssymb}
\usepackage{amsmath}
\usepackage{amsthm}
\usepackage{enumerate}
\usepackage{url}
\usepackage{framed}
\usepackage{hyperref}

\usepackage{draftwatermark}
\SetWatermarkScale{3}
\theoremstyle{definition}
\newtheorem{defi}{Definition}[section]
\newtheorem{example}[defi]{\textit{Example}}

\theoremstyle{plain}
\newtheorem{thm}[defi]{Theorem}
\newtheorem{lemm}[defi]{Lemma}

\newtheorem{prop}[defi]{Proposition}

\theoremstyle{remark}
\newtheorem{rem}[defi]{Remark}

\newenvironment{theorem}{\begin{framed}\begin{thm}}{\end{thm}\end{framed}}
\newenvironment{lemma}[1][]{\begin{framed}\begin{lemm}[#1]}{\end{lemm}\end{framed}}
\newenvironment{remark}{
\begin{rem}}{\end{rem}
}
\newenvironment{definition}{
\begin{defi}}{\end{defi}
}
\newenvironment{proposition}{\begin{framed}\begin{prop}}{\end{prop}\end{framed}}

\numberwithin{equation}{section}

\newcommand{\Bo}{\mathcal{B}}
\newcommand{\C}{\mathbb{C}}
\newcommand{\N}{\mathbb{N}}
\newcommand{\R}{\mathbb{R}}
\newcommand{\ide}{\operatorname{I}}
\newcommand{\Semi}{$C_0$-semigroup }
\newcommand{\DA}{D(A)}
\newcommand{\DAt}{D(A^{2})}
\newcommand{\CL}{C_{\Lambda}}
\newcommand{\gL}{g_{\Lambda}(A)}
\newcommand{\gLo}{g_{1,\Lambda}(A)}
\newcommand{\gLt}{g_{2,\Lambda}(A)}
\newcommand{\gLot}{(g_{1}g_{2})_{\Lambda}(A)}
\newcommand{\RlA}{R(\lambda,A)}
\newcommand{\RmA}{R(\mu,A)}
\newcommand{\Ltwo}{L^{2}(0,\infty)}
\newcommand{\LX}{L^2((0,\infty),X)}
\newcommand{\LH}{L^2((0,\infty),H)}
\newcommand{\LY}{L^2((0,\infty),Y)}
\newcommand{\Lap}{\mathfrak{L}}
\newcommand{\Hinf}{\mathcal{H}_{-}^{\infty}}
\newcommand{\Hinfone}{\mathcal{H}_{-,1}^{\infty}}
\newcommand{\Ht}{\mathcal{H}^{2}}
\newcommand{\Hto}{\mathcal{H}_{\perp}^{2}}
\newcommand{\Ds}{\mathfrak{D}}
\newcommand{\HB}{\mathcal{H}^{\Bo}}

\begin{document}

\title{Weakly admissible $\mathcal{H}^{\infty}(\C_{-})$-calculus on general Banach spaces}

\author{Felix L. Schwenninger\thanks{corresponding author, f.l.schwenninger@utwente.nl} \footnote{Dept. of Applied Mathematics, University of Twente, The Netherlands}
\and Hans Zwart\footnotemark[\value{footnote}]}
\date{July 17, 2012}
\maketitle
\section*{Abstract}
We show that, given a Banach space and a generator of an exponentially stable $C_{0}$-semigroup, a weakly admissible operator $g(A)$ can be defined for any $g$ bounded, analytic function on the left half-plane. This yields an (unbounded) functional calculus. The construction uses a Toeplitz operator and is motivated by system theory. In separable Hilbert spaces, we even get admissibility.  Furthermore, it is investigated when a bounded calculus can be guaranteed. For this we introduce the new notion of exact observability by direction. Finally, it is shown that the calculus coincides with one for half-plane-operators.\newline
This is an extension of the paper\footnote{F. Schwenninger, H. Zwart. \textit{Weakly admissible $\Hinf$ calculus on reflexive Banach spaces} in press, Indag. Math. DOI:10.1016/j.indag.2012.04.005, 2012.}.
\medskip\\

\textit{Keywords:} $H^{\infty}$ functional calculus, Operator semigroup, Toeplitz operator, Weak admissibility, Exact observability by direction

\section{Introduction}

In operator theory we encounter the task of `evaluating' a (scalar-valued) function $f$ where the argument is the operator $A$. 
Simple examples are polynomials, such as the square $A^{2}$, or rational functions, such as, $(\alpha\ide- A)^{-1}$ with $\alpha\in\C$. 
Functional calculus is the field that covers the assignment $f\mapsto f(A)$ for given classes of operators and functions. Beginning with the calculus for self-adjoint operators by von Neumann, \cite{vonneumann}, many classes of operators and functions have been investigated. In general, a  functional calculus should extend a homomorphism which maps from an algebra of functions to the linear space of operators. Furthermore, it should be consistent with the `classical' definitions of rational functions.
Our goal is to construct a  calculus for functions in $\Hinf$, i.e., functions which are bounded and analytic on the left half-plane of $\C$. For the operator $A$, we take a generator of an exponentially stable $C_{0}$-semigroup. The interest for this class lies in numerical analysis and system theory. \newline 
Let us consider the Toeplitz operator $M_{g}:\Ltwo\rightarrow\Ltwo$ with symbol $g\in\Hinf$. By definition, $M_{g}f=\Lap^{-1}\Pi(g\cdot\Lap(f))$, where $\Lap$ is the Laplace transform and $\Pi$ denotes the projection onto $\Ht$, the Hardy space on the right half plane. Since for fixed $a<0$
 \begin{equation*}
 g(s)\cdot\Lap(e^{at})(s)=\frac{g(s)}{s-a}=\frac{g(a)}{s-a}+\frac{g(s)-g(a)}{s-a},
 \end{equation*}
where the last sum is an orthogonal decomposition in $\Ht$ and $\Hto$, we conclude that
 \begin{equation}\label{toeplitzproperty}
 M_{g}(e^{at})=g(a)e^{at}.
 \end{equation}
 In system theoretical words, `exponential input yields exponential output'. Obviously, $g\mapsto g(a)$ is a homomorphism. Our idea is to replace the exponential by the semigroup $e^{At}=T(t)$. In fact,  we show that the formally defined function
 \begin{equation*}
 y(t)=M_{g}(T(.)x_{0})(t)
\end{equation*}
can be seen as the \textit{output} of the linear system
\begin{align*}
\dot{x}(t)={}&Ax(t),\qquad x(0)=x_{0}\\
y(t)={}&Cx(t)
\end{align*}
for some (unbounded) operator $C$. Thus, formally $y(t)=CT(t)x_{0}$. This means that $C$ takes the role of $g(a)$ in (\ref{toeplitzproperty}). Hence, the task is to find $C$ given the \textit{output mapping} $x_{0}\mapsto y(t)$. By G. Weiss, \cite{WeissAdmissible}, this can be done uniquely, incorporating the notion of \textit{admissibility}, see Lemma \ref{weiss}.
\newline The work for separable Hilbert spaces by Zwart,  \cite{ZwartAdmissible}, serves as the main motivation. The aim of this paper is to give a general approach for reflexive Banach spaces. The lack of the Hilbert space structure leads to a weak formulation which will be introduced in Section \ref{sec:calcBanach}. In general, this yields a calculus of \textit{weakly admissible} operators. Then, we turn to the task of giving sufficient conditions  on $A$ that guarantee  bounded $g(A)$ for all $g\in\Hinf$, Section \ref{sec:boundedcalc}. In Subsection \ref{sec:obsvobbd} a connection to the results for the `strong' calculus from \cite{ZwartAdmissible} is established and we see that the weak approach extends the separable Hilbert space case.\newline
Section \ref{sec:furtherresults} is intended to round up the paper by investigating the relation to the calculus for half-plane operators and to give some additional properties. 
 \subsection{The natural $\mathcal{H}^{\infty}$-calculus}\label{naturalcalc}
 Not only in the view of system theory, the class of bounded analytic  functions has attracted much interest in functional calculus in the last decades. Early work was done by McIntosh, \cite{mcintoshHinf}, or can be found for instance in \cite{cowlingdoustmcintoshyagi}. The considered operators are \textit{sectorial} and the main idea is to extend the \textit{Riesz-Dunford}-calculus. We refer to the book by Haase, \cite{haasesectorial}, for an extensive overview. For the generator $A$ of an exponentially stable semigroups, $-A$ is sectorial of angle $\pi/2$, hence, there exists a \textit{natural} (sectorial) calculus (for $A$) for bounded, analytic functions on a larger sector (containing the left half plane). However, since the spectrum of $A$ lies in a half-plane bounded away from the imaginary axis, the appropriate notion is the one of a \textit{half-plane operator} which has been studied in \cite{battyhaasemubeen},  \cite{haasehalfplaneoperators} and \cite{mubeenPhD}. In this context, there exists a \textit{half-plane}-calculus on $\Hinf$ for $A$ being the generator of an exponentially stable semigroup. A brief introduction will be given in Section \ref{subsection:Relationnaturalcalc}.\newline
 In general, it is not clear whether an $\mathcal{H}^{\infty}$-calculus is unique. At least if it is bounded and shares some continuity property, this can be guaranteed, see page 116 in \cite{haasesectorial}.
 \subsection{Setting}
In the following paragraph, we state the setting and recall some notions we are going to use.\\
Let $(X,\|.\|)$ be a Banach space and denote its dual by $X'$. For $x\in X$, $y\in X'$, let $\langle y,x\rangle=\langle x,y\rangle_{X,X'}= y(x)$. If $X_{2}$ is also a Banach space,  the Banach algebra of bounded operators from $X$ to $X_{2}$ is denoted by $\Bo(X,X_{2})$ (or $\Bo(X)$ if $X=X_{2}$). Let $T(.)$ be an exponentially stable \Semi with growth bound  $\omega$. $A$ denotes the generator of $T(.)$. The Banach space $\DA$ equipped with the graph norm of $A$ will be referred to by $(X_{1},\|.\|_{1})$.  For an extensive introduction to semigroups we refer to the book of Engel and Nagel, \cite{engelnagel}.  By $\Lap$ we denote the Laplace transform and by $\Hinf$ we refer to the Banach algebra of bounded holomorphic (complex-valued-)functions on the left half-plane $\C_{-}=\{z\in\C:\Re(z)<0\}$. For $Y$ a Hilbert space, the Hardy spaces $\Ht(Y)$ and $\Hto(Y)$  are defined as the Laplace transforms of $\LY$ and $L^{2}((-\infty,0),Y)$, respectively. One can identify the elements in $\Ht(Y),\Hto(Y)$ with their (limit) boundary functions on the imaginary axis. These limit functions are square integrable and there exists an orthogonal projection $\Pi_{Y}:L^{2}(i\R,Y)\rightarrow \Ht(Y)$ onto $\Ht(Y)$ with kernel $\Hto(Y)$.  For $Y=\C$ we write $\Ht=\Ht(\C)$, $\Pi=\Pi_{\C}$ and so on. Similarly, elements of $\Hinf$ can be identified with essentially bounded functions on $i\R$.\newline
In the following let $\sigma_{\tau}:\LY\rightarrow\LY$, $\tau\geq0$, denote the left shift,
\begin{equation}\label{sigmatau}
		\sigma_{\tau}f=f(.+\tau).
\end{equation} 
	\begin{definition}Let $Y$ be a Hilbert space.
	A linear function $\mathfrak{D}:X\rightarrow \LY$ is called an  \textbf{output mapping} for the \Semi $T(.)$ if
     		\begin{itemize}
		\item $\mathfrak{D}$ is bounded,
		\item 
		for all $\tau\geq0$ and $x\in X$
		    \begin{equation}\label{propsigmatau}
		       \sigma_{\tau}(\mathfrak{D}x)=\mathfrak{D}(T(\tau)x).
		     \end{equation}
		 \end{itemize}
       \end{definition}
All well posed output mappings that we are going to use correspond to the semigroup $T(.)$. In system theory this notion is often named \textit{well-posed infinite-time output mapping}.
\newline
Next, we state a result of G. Weiss, \cite{WeissAdmissible}, which is fundamental for the construction of our functional calculus.
\begin{lemma}[Weiss]
 \label{weiss}Let $Y$ be a Hilbert space.
   For the output mapping $\Ds :X\rightarrow \LY$ there exists a unique $C\in\Bo(X_1,Y)$ and such that
   \begin{equation*}
   			CT(.)x=\Ds x,
   \end{equation*}
   for all $x\in\DA$. This implies that $C$ is \textbf{admissible}, i.e., $\exists m_{1}> 0$ such that
   	\begin{equation*}
             \|CT(.)x\|_{\LY}\leq m_{1}\|x\| \qquad \forall x\in\DA.
          \end{equation*}
 \end{lemma}
 In order to use the previous lemma, we will define an output mapping via a Toeplitz operator. Therefore, we need the following notions and results which were obtained by Zwart in \cite{ZwartAdmissible}.
 \begin{definition}\label{def:Mg}
 Let $H$ be a separable Hilbert space.
 For a function $g\in\Hinf$, define the Toeplitz operator 
 \begin{equation*}
 			M_g:\LH\rightarrow\LH,f\mapsto \mathfrak{L}^{-1}\Pi_{H}(g\cdot\mathfrak{L}f),
	\end{equation*}
where $\mathfrak{L}^{-1}$ denotes the inverse Laplace transform and $\Pi_{H}$ is the orthogonal projection mentioned in the beginning.
\end{definition}
\begin{lemma}
\label{le:propMg}
Let $H$ be a separable Hilbert space and $g,h\in\Hinf$. Then, the following properties hold:
\begin{enumerate}[i.]
			\item \label{boundofToeplitz}$M_g\in\Bo(\LH)$ and $\|M_{g}\|\leq\|g\|_{\infty}$. \label{propMg1}
			\item $\sigma_{\tau}M_{g}=M_{g}\sigma_{\tau}$ for all $\tau\geq0$. \label{sigmaMg}
			\item $M_{g}B=BM_{g}$ for all $B\in\Bo(H)$, i.e., for all $f\in\LH$
			\begin{equation*}
				M_{g}(Bf)=B(M_{g}f),
			\end{equation*}
			where $(Bf)(t)=B(f(t))$ for all $t\geq0$.
			\item $M_{g\cdot h}=M_{g}M_{h}$.
	\end{enumerate}
\end{lemma}
\begin{proof} See \cite{ZwartAdmissible}.
\end{proof}

 \section{$\Hinf$-calculus on Banach spaces}\label{sec:calcBanach}
 
       In the following let $X$ be a Banach space. Furthermore, let $T(.)$  be an exponentially stable \Semi on $X$ and let $g$ be a function in $\Hinf$.
        \subsection{General weak approach}
         \begin{definition}
       \label{def:B}Let $Z$ be a Banach space.
       A  bilinear map $B:X\times Z \rightarrow \Ltwo$  is called a \textbf{weakly admissible output mapping} for $T(.)$ if it is bounded, i.e., $\exists b>0$ such that
           \begin{equation*}
           \|B(x,y)\|_{\Ltwo}\leq b\|x\|\|y\|_{Z} \qquad \forall x\in X, y\in Z,
           \end{equation*}
           and if it has the following property
           \begin{equation}
           \label{propB}
           \sigma_{\tau}B(x,y)=B(T(\tau)x,y) \qquad \forall \tau>0.
           \end{equation}
        \end{definition}
        An example for $B$ is given by $B(x,y)=\langle y,T(.)x\rangle_{X',X}$ for $x\in X, y\in X'$ with $Z=X'$. 
        This fulfills the assumptions of Definition \ref{def:B} because $T(.)$ is exponentially stable.
        \begin{definition}
       For a weakly admissible $B$ we define for fixed $g\in\Hinf,y\in Z$
       \begin{equation}
          \Ds_{g,y}^{B}x=M_{g}(B(x,y))
         \end{equation}
         for $x\in X$ and where $M_{g}\in\Bo(\Ltwo)$ (see Definition \ref{def:Mg} with $H=\C$). 
         \end{definition}
        In the following, let $g$ always be a function in $\Hinf$.
         \begin{lemma}
         \label{le:weakweiss}Let $B$ be a weakly admissible output mapping and $y\in Z$. Then, 
         \begin{equation*}
         \Ds_{g,y}^{B}:X\rightarrow \Ltwo,x\mapsto M_{g}(B(x,y))
         \end{equation*}
          is an output mapping for $T(.)$ and there exists a unique operator 
         $L_{g,y}^{B}\in\Bo (X_{1},\C)$ such that
         
         			\begin{equation}
			\label{eq:weakweiss}
			   \Ds_{g,y}^{B} x_{1}=L_{g,y}^{B}T(.)x_{1}
			 \end{equation}
	for $x_{1}\in \DA$. Furthermore, for $x_{0}\in X$,
		\begin{equation}
		\label{eq:weakweiss2}
		  \Lap[\Ds_{g,y}^{B}x_{0}](s)= L_{g,y}^{B}(s\ide-A)^{-1}x_{0} 
		  \end{equation}
		  on the half-plane $\C_{+}=\left\{z\in\C:\Re(z)>0\right\}$.\newline
		   Hence, for fixed $s\in\C_{+}$,
	                  \begin{align}
	       		  L_{g,y}^{B}x_{1}={}&s\Lap[\Ds_{g,y}^{B}x_{1}](s)-\Lap[\Ds_{g,y}^{B}Ax_{1}](s )\notag\\
			  			   ={}&\int_{0}^{\infty}[\Ds_{g,y}^{B}(s\ide-A)x_{1}](t)e^{-st} \ dt\notag\\
						   ={}&\Lap[\Ds_{g,y}^{B}(s\ide-A)x_{1}](s) \label{LgyisLap}
			\end{align}
	for $x_{1}\in\DA$.           
	 \end{lemma}
	 \begin{proof}
	 By the properties of $B$ and $M_g$ (see Lemma \ref{le:propMg}) we see that $\Ds_{g,y}^{B}$ is an output mapping. In fact,  for fixed $y$, $\Ds_{g,y}^{B}$ is bounded as composition of the bounded operators $M_g\in\Bo(\Ltwo))$ and $B(.,y)$. Furthermore, by Lemma  \ref{le:propMg}.\ref{sigmaMg}., and (\ref{propB})    
	 	 \begin{align*}
	 \sigma_{\tau}\Ds_{g,y}^{B}x={}&\sigma_{\tau}M_{g}(B(x,y))\\
	 					   ={}&M_{g}(\sigma_{\tau}(B(x,y)))\\
						   ={}&M_{g}(B(T(\tau)x,y))\\
						   ={}&\Ds_{g,y}^{B}T(\tau)x
	\end{align*} 
	for all $x\in X$.
	  Now that we know that $\Ds_{g,y}^{B}$ is an output mapping, Lemma \ref{weiss} yields the first assertion.
	 Taking the Laplace transform of (\ref{eq:weakweiss}), which exists for $\Re(s)>0$  since $\Ds_{g,y}^{B}\in\Ltwo$, and, using that the integrals exist in $X_{1}$, we deduce
	 \begin{equation*}
	     \Lap[\Ds_{g,y}^{B}x_{1}](.)=L_{g,y}^{B}(.\ide-A)^{-1}x_{1}
	  \end{equation*}
	  on $\C_{+}$ for $x_{1}\in\DA$. Since $\DA$ is dense, (\ref{eq:weakweiss2}) follows. Fixing $s\in\C_{+}$ and taking $x_{0}=(s\ide-A)x_{1}$ yields the last assertion.
	  \end{proof}
                    
         \begin{lemma}\label{Lem2}The following assertions hold
           \begin{enumerate}[i.]
           	\item For all $x\in X$ and $y\in Z$
           		\begin{equation}
           		\label{ineq:Dgy}
	               	\|\Ds_{g,y}^{B}x\|_{\Ltwo}\leq b\|g\|_{\infty}\|y\|_{Z}\|x\|.
           		\end{equation}
		\item\label{weakConvOutMap} If $Z=X'$ and $B(x,y)=\langle y,T(.)x\rangle$ for $x\in X,y\in X'$, then, for a sequence $y_{n}\in X',n\in\N$,
		\begin{equation*}
		\Big(\langle y_{n},x\rangle\stackrel{n\to\infty}{\longrightarrow}\langle y,x\rangle\ \forall x\in X\Big)\Longrightarrow \Big(\Ds_{g,y_{n}}^{B}x\stackrel{n\to\infty}{\longrightarrow}\Ds_{g,y}^{B}x\  \forall x\in X\Big), 
		\end{equation*}
		that is, for fixed $x\in X$, the mapping $y\mapsto\Ds_{g,y}^{B}x$ is $weak*$-continuous.
			             \end{enumerate}
		\end{lemma}
		\begin{proof}
			The first assertion follows immediately from the definition of $\Ds_{g,y}^{B}$, part \ref{propMg1}.\ of 
			Lemma \ref{le:propMg}, and the boundedness of $B$. 
			To see \ref{weakConvOutMap}., assume that $y_{n}\in X'$ converges to $y$ in the $weak*$-topology. 
			Therefore, for fixed $x\in X$,
			\begin{equation*}
				\langle y_{n}-y,T(t)x\rangle_{X',X}\stackrel{n\to\infty}{\longrightarrow}0 \qquad\forall t\geq0.
			\end{equation*}
			Since $weak*$-convergence implies that $y_{n}$ is bounded (by uniform boundedness principle), 
			there exists a constant $\tilde{M}>0$ such that
			\begin{equation*}
			|\langle y_{n}-y,T(t)x\rangle|\leq\tilde{M}\cdot\|T(t)x\|\leq\tilde{M}e^{\omega t}\qquad \forall t\geq0,
			\end{equation*}
			where $\omega<0$ is the growth bound of the semigroup. 
			Hence, dominated convergence yields that $\langle y_{n}-y,T(t)x\rangle\to0$ in $\Ltwo$ as $n\to\infty$.
			By definition, $\Ds_{g,y}^{B}x=M_{g}\langle y,T(t)x\rangle$, and since $M_{g}\in\Bo(\Ltwo)$ we conclude
			\begin{equation*}
				\Ds_{g,y_{n}}^{B}x\stackrel{n\to\infty}{\longrightarrow}\Ds_{g,y}^{B}x.
			\end{equation*}
		\end{proof}
 Using the lemma above, we can deduce properties of the mapping $y\mapsto L_{g,y}^{B}x$.
 	
		\begin{lemma} \label{Lem3}The following assertions hold
		\begin{enumerate}[i.]
		\item \label{lemCx1} There exists $b_{2}>0$ such that
                		\begin{equation}
               		 \label{ineq:Lgy}
	            	   |L_{g,y}^{B}x_{1}|\leq b_{2}\|g\|_{\infty}\|y\|_{Z}\|x_{1}\|_{1}\qquad x_{1}\in\DA,y\in Z.
           		\end{equation}
		   	\item \label{lemCx2} 	For fixed $x_{1}\in \DA$ the mapping
             		\begin{equation*}
                 	 L_{g,.}^{B}x_{1}:Z\rightarrow \C, \qquad y\mapsto L_{g,y}^{B}x_{1}
             		\end{equation*}
             is linear and bounded, hence in $Z'$, i.e. there exists a unique element 
             $f_{x_{1}}$ in $Z'$ such that 
             \begin{equation}
             \label{Lgyfx}
                   L_{g,y}^{B}x_{1}=\langle y,f_{x_{1}}\rangle  _{Z,Z'}	\qquad \forall y\in Z.
              \end{equation}

             		\item \label{lemCx3}
				If $Z=X'$ and $B(x,y)=\langle y,T(.)x\rangle$ for $x\in X,y\in X'$, then for every $x_{1}\in\DA$ there exists a unique element $f_{x_{1}}\in X$ such that 
				\begin{equation} 
				\label{LgyfxX}
				 L_{g,y}^{B}x_{1}=\langle y,f_{x_{1}}\rangle  _{X',X}	\qquad \forall y\in X'.
             			 \end{equation}
					
             \end{enumerate}
                 \end{lemma}
                
           \begin{proof}
           	For \ref{lemCx1}., fix an $s$ with $\Re(s)>0$. Note that by Cauchy-Schwarz and (\ref{ineq:Dgy})
	    \begin{align*}
	     |\int_{0}^{\infty}[\Ds_{g,y}^{B}(s\ide-A)x_{1}](t)e^{-st} \ dt|\leq{}& b_{s}\|\Ds_{g,y}^{B}(s\ide-A)x_{1}\|_{\Ltwo}\\
	     										\leq{}&b_{s}b\|g\|_{\infty}\|y\|_{Z}\|(s\ide-A)x_{1}\|
	     \end{align*}
	     for some constant $b_{s}>0$. By Lemma \ref{le:weakweiss}, the left hand side equals $|L_{g,y}^{B}x_{1}|$ and we obtain (\ref{ineq:Lgy}) because  $(s\ide-A)\in\Bo(X_{1},X)$. Having  \ref{lemCx1}., for  \ref{lemCx2}.\ it remains to show the linearity of  $L_{g,.}x_{1}$ for fixed $x_{1}\in\DA$. By the linearity of $B(x_{0},.)$ and $M_{g}$ it is clear that $\Ds_{g,.}x_{0}$ is linear,  for fixed $x_{0}\in X$. Hence, using (\ref{LgyisLap}) again for some fixed $s\in\C_{+}$, we have for $y,z\in Z$ and $\lambda\in\C$
	     \begin{align*}
	     L_{g,y+\lambda z}x_{1}={}&\Lap[\Ds_{g,y+\lambda z}^{B}(s\ide-A)x_{1}](s)\\
	     					={}&\Lap[\Ds_{g,y}^{B}(s\ide-A)x_{1}](s)+\lambda\Lap[\Ds_{g,z}^{B}(s\ide-A)x_{1}](s)\\
						={}&L_{g,y}x_{1}+\lambda L_{g,z}x_{1}.
	   \end{align*}
	   In  \ref{lemCx3}.\ we have to show that the $f_{x_{1}}\in Z'=X''$ from \ref{lemCx2}.\ can be identified with an element in $X$.
	   For that, it suffices to prove that $L_{g,.}x_{1}:X'\rightarrow\C$ is $weak*$-continuous (see for instance \cite[Theorem IV.8.1.]{yosida}). So, let $y_{n}\in X'$ converge to $y\in X'$ in the $weak*$-topology. 
	   From (\ref{LgyisLap}) we have that for $x_{1}\in\DA, s\in\C_{+}$
	   \begin{align*}
	   |L_{g,y_{n}}x_{1}-L_{g,y}x_{1}|={}&|\int_{0}^{\infty}[\Ds_{g,y_{n}}^{B}(s\ide-A)x_{1}-\Ds_{g,y}^{B}(s\ide-A)x_{1}](t)e^{-st} \ dt|\\
	   						\leq{}&b_{s}\cdot\|\Ds_{g,y_{n}}^{B}(s\ide-A)x_{1}-\Ds_{g,y}^{B}(s\ide-A)x_{1}\|_{\Ltwo},
	   \end{align*}
	   where we used Cauchy-Schwarz. By Lemma \ref{Lem2}.\ref{weakConvOutMap}.\ we know that the term on the right hand side goes to zero as $n\to\infty$. Hence, $L_{g,.}x_{1}$ is $weak*$-continuous.
	     \end{proof}
            Having in mind \ref{lemCx2}.\ (or \ref{lemCx3}.) of the previous lemma, we consider the map 
                 \begin{equation}\label{mapg(A)}
                 g^{B}(A):\DA\rightarrow Z'\ (\text{or }X),\qquad x_{1}\mapsto f_{x_{1}}=L_{g,.}^{B}x_{1}
                 \end{equation}
                 It is linear since $L_{g,y}^{B}x_{1}$ is linear in $x_{1}$ and by (\ref{ineq:Lgy}) it is bounded, i.e., $g^{B}(A)\in\Bo(X_{1},Z')$ (or $\Bo(X_{1},X)$). Now, we are able to state the main result of the general weak approach.
                             
        \begin{theorem}\label{thm:g(A)}
                 Let $A$ be the generator of  an exponentially stable \Semi $T(.)$ and let $B:X\times Z \rightarrow \Ltwo$ be a weakly admissible output mapping. Then for $g\in\Hinf$ the following assertions hold
             \begin{enumerate}[i.]                 
                  \item \label{thm:g(A)item1}
                 There exists a unique operator $g^{B}(A)\in\Bo(X_{1},Z')$ such that 
                 \begin{equation}
                 \label{eq:thm:g(A)}
                       \Ds_{g,y}^{B}x_{1}=\langle y,g^{B}(A)T(.)x_{1}\rangle  _{Z,Z'}
                 \end{equation}      
                for all $y\in Z$ and $x_{1}\in\DA$. 
                 \item  \label{thm:g(A)item4}If $Z=X'$ and $B(x,y)=\langle y,T(.)x\rangle$ for $x\in X,y\in X'$, 
                 	  then the operator $g^{B}(A)$ is in $\Bo(X_{1},X)$ and  
                 \begin{equation}
                 \label{eq:thm:g(A)X}
                       \Ds_{g,y}^{B}x_{1}=\langle y,g^{B}(A)T(.)x_{1}\rangle  _{X',X}
                 \end{equation}      

                \item There exists a constant $\alpha>0$ such that for all $x\in X$
                \begin{equation}\label{weakWeiss}
                \left\|g^{B}(A)(s\ide-A)^{-1}x\right\|_{Z'}\leq\frac{\alpha}{\sqrt{\Re(s)}}\|g\|_{\infty}\|x\|.
                \end{equation}
                \item \label{thm:g(A)2}
                If $Z=X'$ and 
                \begin{equation}
                \label{commuteBT}
                	B(T(t)x,y)=B(x,T'(t)y)\qquad \text{for all }t\geq0,x\in X, y\in X'
	      \end{equation}
	      then 
                 $g^{B}(A)$ commutes with the semigroup, i.e., 
                  \begin{equation}
                  \label{g(A)commuteT}
                       \langle y,g^{B}(A)T(t)x_{1}\rangle  _{X',X''}=\langle T(t)'y,g^{B}(A)x_{1}\rangle  _{X',X''} ,
                  \end{equation}
                  for $x_{1}\in \DA$, $y\in X'$ and all $t\geq0$.
                \end{enumerate}
                \end{theorem}
             
                \begin{proof}
                The first assertion follows by (\ref{eq:weakweiss}) and the considerations above, see Lemma \ref{Lem3}.\ref{lemCx2}.\ and (\ref{mapg(A)}), from which we have that               
                \begin{equation*}
                L_{g,y}^{B}T(.)x=\langle y,g^{B}(A)T(.)x\rangle_{Z,Z'}.
                \end{equation*}
                Part \ref{thm:g(A)item4}.\ follows as in \ref{thm:g(A)item1}.\ but using \ref{Lem3}.\ref{lemCx3}. instead of \ref{Lem3}.\ref{lemCx2}.\newline
                 Inequality (\ref{weakWeiss}) is a consequence of (\ref{eq:weakweiss2}). In fact, for $\Re(s)>0$ we have by Cauchy-Schwarz
                 \begin{align*}
                 \left|\langle y,g^{B}(A)(s\ide-A)^{-1}\rangle _{Z,Z'}\right|={}&\left|\Lap[\Ds_{g,y}^{B}x](s)\right|\\
                 								    \leq{}&\frac{1}{\sqrt{\Re(2s)}}\left\|\Ds_{g,y}^{B}x\right\|_{\LX}\\
								    	    \leq{}&\frac{\alpha}{\sqrt{\Re(s)}}\|g\|_{\infty}\|x\|\|y\|_{Z},
	       \end{align*}
	       where in the last step we used the boundedness of the output mapping,  (\ref{ineq:Dgy}).\newline
                To see (\ref{g(A)commuteT}), we use (\ref{LgyisLap}) and (\ref{Lgyfx}). Let $t> 0$, $\Re(s)>0$, $y\in X'$ and $x_{1}\in\DA$. Then,
                 \begin{align}
                  \langle y,g^{B}(A)T(t)x_{1}\rangle  _{X',X''}={}&L_{g,y}^{B}T(t)x_{1}\notag\\
                        			        ={}&\Lap[\Ds_{g,y}^{B}(s\ide-A)T(t)x_{1}](s)\notag\\
			                          ={}&\Lap[M_{g}B(T(.)(s\ide-A)T(t)x_{1},y)](s)\label{proofThmg(A)}
	       \end{align}
	       By exploiting the additional assumption on $B$, (\ref{commuteBT}), we deduce further
	       \begin{align*}
	         \Lap[M_{g}B(T(.)(s\ide-A)T(t)x_{1},y)](s)={}&\Lap[M_{g}B(T(.)(s-A)x_{1},T'(t)y)](s)\\
						                         ={}&\Lap[\Ds_{g,T'(t)y}(s-A)x_{1}](s)\\
			         				                ={}&L_{g,T'(t)y}x_{1}\\
			                          			       ={}&\langle T(t)'y,g^{B}(A)x_{1}\rangle  _{X',X''}.
	         \end{align*}
	         Together with (\ref{proofThmg(A)}), this gives the assertion.\newline
		\end{proof}	
		                          
                 Theorem \ref{thm:g(A)} and estimate (\ref{ineq:Dgy}) motivate the introduction of the following notion.
                 \begin{definition}Let $Y$ be a Banach space.
                 An operator $C\in\Bo(X_{1},Y)$ is called \textbf{weakly admissible} if there exists an $m> 0$ such that for all $x\in\DA$ and $y\in Y'$
                  \begin{itemize}
                      \item $\langle y,CT(.)x\rangle  \in\Ltwo$ and
                      \item $\|\langle y,CT(.)x\rangle  \|_{\Ltwo}\leq m\|y\|_{Y'}\|x\|$.
                  \end{itemize}
                  \end{definition}
                  \begin{remark}\label{rem:B=CT}
                  \begin{itemize}\item
                  From this definition we get immediately that if $C\in\Bo(X_{1},Y)$ is weakly admissible, then $\tilde{B}(x,y)=\langle y,CT(.)x\rangle_{Y',Y} $ defined on $\DA\times Y'$ can be uniquely extended to a  bilinear mapping $B$ on $X\times Y'$. This $B$ fulfills the assumptions in Definition \ref{def:B} ($Z=Y'$) and because of this, $\Ds_{g,y}^{C},L_{g,y}^{C},g^{C}(A)$ will denote $\Ds_{g,y}^{B},L_{g,y}^{B},g^{B}(A)$ respectively. Note that this $B$ does not satisfy (\ref{commuteBT}) in general even if $Y=X'$.
                  \item From Theorem \ref{thm:g(A)} and (\ref{ineq:Dgy}), it follows that $g^{B}(A)$ is weakly admissible.
                     \end{itemize}
		\end{remark}
                  \begin{remark}\label{remendoffirstsection}
                  The notion of weak admissibility  and its connection to (strong) admissibility has been investigated for instance by Weiss who conjectured that the terms are equivalent. However, even for Hilbert spaces counterexamples were found, see \cite{ZwartJacobStaffans}, \cite{JacobPartingtonPott}.
                  \end{remark}

\subsection{The calculus}

              In the following we will set $Z=X'$. 
              For the rest of the paper, $g(A)$ will denote $g^{B}(A)$ for the weakly admissible mapping 
              $B(x,y)=\langle y,T(.)x\rangle  _{X',X}$. 
              Consequently, we will write $\Ds_{g,y}$ and $L_{g,y}$ when this specific $B$ is meant.
%
%
              We are going to need the following lemmata several times. 
              \begin{lemma}The operator 
              $g(A)$ is a bounded operator from $X_{1}$ to $X$ which commutes with the semigroup, i.e.,
                          \begin{equation}
                          \label{g(A)commuteT2}
                           g(A)T(t)=T(t)g(A)
                          \end{equation}
                          on $\DA$ for all $t>0$. Therefore, for $\lambda\in\rho(A)$
                          \begin{equation}
                          \label{g(A)commuteA}
                          g(A)\RlA x_{1}=\RlA g(A)x_{1} \qquad \forall x_{1}\in\DA.
                          \end{equation}
                          In particular, $g(A)\DAt\subset\DA$.
               \end{lemma}
               \begin{proof}
               The first assertions follow all directly from Theorem \ref{thm:g(A)}. To see (\ref{g(A)commuteA}) consider the Laplace transform of equation (\ref{g(A)commuteT2}).
               \end{proof}                      
              \begin{lemma}
              \label{le;compositewadm}
                Let $C\in\Bo(X_{1},Y)$ be weakly admissible. Then $Cg(A)$ is weakly admissible (in the sense that it can be extended uniquely to a weakly admissible operator from $X_{1}$ to $Y$) and
                \begin{equation}
                		Cg(A)x_{2}=g^{C}(A)x_{2}\qquad \forall x_{2}\in\DAt
	      \end{equation}
	      where $g^{C}(A)$ is the operator from (\ref{eq:thm:g(A)}) with $B(x,y)=\langle y,CT(.)x\rangle _{Y',Y}$ (see Remark \ref{rem:B=CT}).
	      \end{lemma}
	      
	      \begin{proof}
	      Let $x\in\DAt$ and $y\in Y'$.  Then $Ax\in\DA$. Using (\ref{g(A)commuteA}) and that $CA^{-1}\in\Bo(X,Y)$, we obtain 
	      \begin{align}
	      \langle y,Cg(A)T(t)x\rangle_{Y',Y}  ={}&\langle y,CA^{-1}g(A)T(t)Ax\rangle_{Y',Y} \notag \\
	      		       				={}&\langle (CA^{-1})'y,g(A)T(t)Ax\rangle_{X',X}  \notag\\
							={}&\big(\Ds_{g,(CA^{-1})'y}(Ax)\big)(t)\notag\\
	                          				={}&\big(M_{g}(\langle (CA^{-1})'y,T(.)Ax\rangle_{X',X})\big)(t)\notag\\
	                          				={}&\big(M_{g}(\langle y,CT(.)x\rangle_{Y',Y})\big)(t)\label{eq:lemCg(A)}\\
	                          				={}&\big(\Ds_{g,y}^{B}x\big)(t)\notag\\
	                          				={}&\langle y,g^{C}(A)T(t)x\rangle_{Y',Y}.\notag
	      \end{align}
	     The equality holds for all $t\geq0$ point-wise since both the right and the left hand-side are continuous functions  for $x\in\DAt$. 
	      \end{proof}
	      As pointed out in Remark \ref{rem:B=CT}, $g^{C}(A)$ will not commute with the semigroup in general. However, if $C\in\Bo(X_{1},X)$ commutes with $T(.)$, then 
	      \begin{equation*}
	      B(T(t)x,y)=\langle y,CT(.)T(t)x\rangle_{X',X}=\langle T'(t)y,CT(.t)x\rangle_{X',X}=B(x,T'(t)y)
		      \end{equation*}
	      for all $t\geq0$ and $x\in X$. Hence, by Theorem \ref{thm:g(A)}.\ref{thm:g(A)2}., we conclude that $g^{C}(A)T(t)=T(t)g^{C}(A)$ for all $t\geq0$ in this case.\par
	      It may happen that $g(A)$ is even bounded in the norm of $X$. 
	      However, still it will be defined only on $\DA$ by the construction. 
	      Then, we would like to identify it with its bounded extension to $X$.
	      Moreover, although $g(A)$ will not be bounded in $X$ in general, it can be extended to a closed operator. 
	      For that, we introduce the following extension.
                     \begin{lemma}
                     Let $C$ be an operator in $\Bo(X_{1},X)$ which commutes with some (any) resolvent $\RmA=(\mu\ide-A)^{-1}$. Then, the operator
                     \begin{align}
                         &\CL x=\lim_{\lambda\to\infty}\lambda C\RlA x\label{LebExt}\\
                         &D(\CL)=\{x\in X:\text{the above limit exists}\}\notag
                     \end{align}
                     is a closed extension of $C$. $\CL$ commutes with any resolvent $\RmA$ on $D(\CL)$. This operator is called the \textbf{Lambda extension}. If $C$ is bounded in $X$, then $C_{\Lambda}\in\Bo(X)$. 
                     \end{lemma}
                     \begin{proof}Recall the following property of a \Semi (see Lemma 3.4 in \cite{engelnagel})
                     \begin{equation}\label{semipropconv}
                     \lim_{\lambda\to\infty}\lambda\RlA x=x\qquad \forall x\in X.
                     \end{equation}
                     First, let $x\in\DA$. By assumption, $\lambda C\RlA x=\lambda\RlA Cx$ which converges to $Cx$ as $\lambda\to\infty$.  Thus, $\CL$ is an extension of $C$. Now, let $x\in D(\CL)$ and $\mu\in\rho(A)$. Since $\RmA$ is bounded and $C\RmA x= \RmA Cx$ on $\DA$ by assumption, we have
                     \begin{align}
                    \RmA\CL x={}&\lim_{\lambda\to\infty}\lambda\RmA C\RlA x\notag\\
                    ={}&\lim_{\lambda\to\infty}\lambda C\RlA\RmA x= \CL\RmA x.\label{le:LE2}
                     		   		   \end{align}
		Hence, we have proved that $\CL$ commutes with the resolvent. Next, we show that it is a closed operator.
		 Let $\left\{x_{n}\right\}$ be a sequence in $D(\CL)$ such that $x_{n}\to x$ and $\CL x_{n}\to z$ for $n\to\infty$. By (\ref{le:LE2})  and since $\RmA x_{n}\in\DA$,
		 \begin{equation*}
		 \RmA\CL x_{n}=\CL\RmA x_{n}=C\RmA x_{n}
		 \end{equation*}
		 for all $n\in\N$. Since $C\RmA\in\Bo(X)$, we deduce for the limit $n\to\infty$
		 \begin{equation*}
		 \RmA z=C\RmA x.
		\end{equation*}
		Multiply by $\mu$ and let $\mu\to\infty$. By (\ref{semipropconv}) the limit exists and 
		\begin{equation*}
		z=\lim_{\mu\to\infty}\mu C\RmA x
		\end{equation*}
		holds. Thus, $x\in D(\CL)$ and $\CL x=z$.\newline
		If $C$ is bounded in $X$, then there exists a unique extension $\overline{C}\in\Bo(X)$, $C\subset\overline{C}$. By (\ref{semipropconv}), it follows that $\CL=\overline{C}$.   
                      \end{proof}

		In the following let $\gL$ denote the Lambda extension of $g(A)$. We make the convention that for (unbounded) operators $F, G$ the domain of $F+G$ is $D(F)\cap D(G)$.

                  \begin{theorem}\label{mainthm}
              $g\mapsto \gL$ fulfills the properties of an (unbounded) functional calculus, i.e.,
            \begin{enumerate}[i)]
               \item $g\equiv 1\Rightarrow \ g(A)=\ide$,
               \item \label{mainthm2}$(g_{1}+g_{2})_{\Lambda}(A)\supset \gLo+\gLt$, 
               \item \label{mainthm3}$\gLot\supset \gLo\gLt$ and 
               \begin{equation}\label{domincl}
               D\big(\gLo\gLt\big)=D\big(\gLot\big)\cap D\big(\gLt\big).
               \end{equation}
            \end{enumerate}
            If $g_{2}(A)$ is bounded, then equality holds in \ref{mainthm2}) and \ref{mainthm3}).
                    \end{theorem}
                      \begin{proof}
                      Obviously, for $g\equiv1\in\Hinf$, $\Ds_{g,y}f=f$ and thus, $g(A)=\ide$. Since the Toeplitz operator $M_{g}$ is linear in symbol $g$, it follows that
                      \begin{equation*}
                      (g_{1}+g_{2})(A)=g_{1}(A)+g_{2}(A)
                      \end{equation*}
                      defined on $\DA$. For $x\in D(g_{1,\Lambda}(A)+g_{2,\Lambda}(A))=D(g_{1,\Lambda}(A))\cap D(g_{2,\Lambda}(A))$ it follows that
                      \begin{equation}\label{mainthmproof1}
                      \lim_{\lambda\to\infty}\lambda(g_{1}(A)+g_{2}(A))\RlA x 
                      \end{equation}
                      exists. Hence, $x$ lies in the domain of $(g_{1}+g_{2})_{\Lambda}(A)$. If $g_{2}(A)$ is bounded, then $D(\gLt)=X$. Thus, the existence of (\ref{mainthmproof1}) implies that 
                      $x\in D(\gLo)$.\newline
                      In order to show \ref{mainthm3}), we verify $(g_{1}\cdot g_{2})(A)=g_{1}(A)g_{2}(A)$ on $\DAt$ first. 
                      According to Lemma \ref{le;compositewadm}, it suffices to prove $g_{1}^{C}(A)=(g_{1}\cdot g_{2})(A)$ for $C=g_{2}(A)$. Let $y\in X'$ and $x\in\DAt$.
                      Then,                      
                      \begin{align*}
                      \langle y,(g_{1}g_{2})(A)T(t)x\rangle  ={}&\big(\Ds_{(g_{1}g_{2}),y}x\big)(t)\\
                      				        ={}&\big(M_{g_{1}g_{2}}(\langle y,T(.)x\rangle  )\big)(t)\\
                               			        ={}&\big(M_{g_{1}}M_{g_{2}}(\langle y,T(.)x\rangle  )\big)(t)\\
			           		        ={}&\big(M_{g_{1}}(\Ds_{g_{2},y}x)\big)(t)\\
					                 ={}&\big(M_{g_{1}}(\langle y,g_{2}(A)T(.)x\rangle  )\big)(t)\\
					                 ={}&\langle y,g_{1}^{C}(A)T(t)x\rangle  ,
		    \end{align*}
			where we used (\ref{eq:thm:g(A)}) several times as well as  the fact that $M_{g_{1}g_{2}}=M_{g_{1}}M_{g_{2}}$ (see Lemma \ref{le:propMg}). Since $x\in\DAt$, the equality holds point-wise for $t\geq0$.  Thus, 
			\begin{equation}\label{thm:coincidenceonDAt}
			   (g_{1}\cdot g_{2})(A)x_{2}=g_{1}(A)g_{2}(A)x_{2}\qquad \forall x_{2}\in\DAt.
			  \end{equation} 
			  Now, let $x\in D(\gLo\gLt)$. This means that
			  \begin{align*}
			  &\lim_{\mu\to\infty}g_{2}(A)\RmA x=\gLt x\quad\text{ exists as well as }\\
			  &\lim_{\lambda\to\infty}\lambda g_{1}(A)\RlA\gLt x=\gLo\gLt x.
			  \end{align*} 
			  Since $g_{1}(A)\RlA\in\Bo(X)$ and since $\RlA$ commutes with  $g_{2}(A)$ on $\DA$, (\ref{g(A)commuteA}), we obtain that
			  \begin{equation*}
			  \gLo\gLt x=\lim_{\lambda\to\infty}\lim_{\mu\to\infty}(\lambda\mu) g_{1}(A)g_{2}(A)\RlA\RmA x
			  \end{equation*}
			  Clearly,  $\RlA\RmA x\in\DAt$. Thus, by (\ref{thm:coincidenceonDAt}),
			    \begin{equation*}
			  \gLo\gLt x=\lim_{\lambda\to\infty}\lim_{\mu\to\infty}(\lambda\mu) (g_{1}g_{2})(A)\RlA\RmA x
			  \end{equation*}
			  Using the resolvent identity, this can be written as
			  \begin{equation}\label{thm:limLaext}
			   \gLo\gLt x=\lim_{\lambda\to\infty}\lim_{\mu\to\infty}\frac{\lambda\mu}{\mu-\lambda} (g_{1}g_{2})(A)\Big[\RlA x-\RmA x\Big].
			  \end{equation}
			  By (\ref{weakWeiss}), we have that $(g_{1}g_{2})(A)\RmA x\to0$ as $\mu\to\infty$. Therefore,
			  \begin{equation*}
		            \lim_{\mu\to\infty}\frac{\lambda\mu}{\mu-\lambda} (g_{1}g_{2})(A)\RmA x=0.
		            \end{equation*}
		            Furthermore,
		            \begin{equation*}
		            \lim_{\mu\to\infty}\frac{\lambda\mu}{\mu-\lambda} (g_{1}g_{2})(A)\RlA x=\lambda  (g_{1}g_{2})(A)\RlA x.
		            \end{equation*}
		            Together, this yields the limit in (\ref{thm:limLaext}),
		            \begin{equation*}
		             \gLo\gLt x=\lim_{\lambda\to\infty}\lambda  (g_{1}g_{2})(A)\RlA x
		            \end{equation*}
		            which means that $x\in D(\gLot)$ and $\gLot x=\gLo\gLt x$. This also shows the inclusion `$\subseteq$' in (\ref{domincl}) since $x\in D(\gLt)$ by assumption. To show the other inclusion, we observe that for $x\in X$ and $\mu\in\rho(A)$
		            \begin{align*}
		            (g_{1}g_{2})(A)\RmA x={}&\lim_{\lambda\to\infty}\lambda (g_{1}g_{2})(A)\RlA\RmA x\\
		            				   ={}&\lim_{\lambda\to\infty}\lambda g_{1}(A)g_{2}(A)\RlA\RmA x\\
						            ={}&\lim_{\lambda\to\infty}\lambda g_{1}(A)\RlA g_{2}(A)\RmA x,
			 \end{align*}
			 where we used (\ref{thm:coincidenceonDAt}) and that $\RlA\RmA x$, $\RmA x$ lie in $\DAt$ and $\DA$, respectively.
			 This gives that $g_{2}(A)\RmA x\in D(\gLo)$ and
			 \begin{equation*}
		           (g_{1}g_{2})(A)\RmA x=\gLo g_{2}(A)\RmA x.
		           \end{equation*}
			For $x\in D(\gLot)\cap D(\gLt)$ this yields that the limit
			\begin{equation*}
			\lim_{\mu\to\infty} \mu(g_{1}g_{2})(A)\RmA x=\lim_{\mu\to\infty}\gLo \mu g_{2}(A)\RmA x
			\end{equation*}
			exists. Since $\mu g_{2}(A)\RmA x\to \gLt$ for $\mu\to\infty$ and the closedness of $\gLo$ we deduce 
			\begin{equation*}
			\gLt x\in D(\gLo)\quad\text{ and } \quad  \gLo \gLt x=\gLot x.
			\end{equation*}
			This shows that $x\in D\big(\gLo\gLt\big)$. For bounded $g_{2}(A)$, (\ref{domincl}) directly shows the equality.
		            
                      \end{proof}
                    Next, we see that our weak calculus coincides with the `usual' definition of $g(A)$ in case of $g$ being rational.
                   \begin{lemma}\label{grational}
                   If $g$ is the Fourier transform of a function $h\in L^{1}(\R)$ with  support in $(-\infty,0]$, then
		\begin{equation}
		  g(A)x=\int_{0}^{\infty}h(-s)T(s)x \ ds
		 \end{equation}
		 for all $x\in\DA$. Hence, $g(A)$ is bounded and can be extended continuously to an operator in $\Bo(X)$.
	           \end{lemma} 
	           \begin{proof}
	          Let $y\in X'$, $x\in X$ and $s>0$. By equation (\ref{eq:weakweiss2}) of Lemma \ref{le:weakweiss} and (\ref{Lgyfx}) we know that
		\begin{equation}\label{eq3:proofg=Fh}
		 \langle y,g(A)(s\ide-A)^{-1}x \rangle_{X',X}=\Lap[\Ds_{g,y}x](s).
		\end{equation}
		We are going to use the following general consequence of the Fourier transform. For $f\in\Ltwo$ it follows by the Convolution Theorem that 
		\begin{align}
				g\cdot\Lap(f)(i.)={}&\mathcal{F}(h)(.)\mathcal{F}(f_{ext})(.)=\mathcal{F}(h\ast f_{ext})(.)\notag \\
		={}&\Lap((h\ast f_{ext})\big|_{(0,\infty)})(i.)+\Lap((h\ast f_{ext})\big|_{(-\infty,0))})(i.),\label{eq:proofg=Fh}
		\end{align}
		where $f_{ext}$ is the extension of $f$ to the real line, by $f_{ext}(t)=0$ for $t<0$. 
		Since $h\ast f_{ext}\in L^{2}(\R)$ by Young's inequality, (\ref{eq:proofg=Fh}) yields
		\begin{equation}\label{eq2:proofg=Fh}
		M_{g}f=\Lap^{-1}\Pi(g\cdot\Lap f)=(h\ast f_{ext})\big|_{(0,\infty)}.
		\end{equation}
		Now, let $f=\langle y,T(.)x\rangle_{X',X}$.  By equation (\ref{eq2:proofg=Fh}),
		\begin{align*}
		\Lap[\Ds_{g,y}x](s)={}&\Lap[M_{g}f](s)\\
					     ={}&\int_{0}^{\infty}e^{-st}(h\ast f_{ext})\big|_{(0,\infty)}(t) \ dt\\
					     ={}&\int_{0}^{\infty}e^{-st}\int_{\R}\langle y,T(t-u)x\rangle h(u) \ du \ dt\\
					     ={}&\int_{0}^{\infty}e^{-st}\int_{0}^{\infty}h(-u)\langle y,T(t+u)x\rangle\  du \ dt\\
					     ={}&\int_{0}^{\infty}\int_{0}^{\infty}\langle y,h(-u)T(u)e^{-st}T(t)x\rangle\ du \ dt\\
					     ={}&\int_{0}^{\infty}\langle y,h(-u)T(u)\int_{0}^{\infty}e^{-st}T(t)x \ dt\rangle\ du\\
					     ={}&\int_{0}^{\infty}\langle y,h(-u)T(u)(s\ide-A)^{-1}x\rangle \ du,
		\end{align*}
		where we used Fubini's Theorem and the fact that $\int_{0}^{\infty}e^{-st}T(t)x \ dt=(s\ide-A)^{-1}x$. 
		Inserting this in (\ref{eq3:proofg=Fh}) gives
		\begin{align*}
		\langle y,g(A)(s\ide-A)^{-1}x \rangle_{X',X}={}&\int_{0}^{\infty}\langle y,h(-u)T(u)(s\ide-A)^{-1}x\rangle_{X',X}\ du\\
								    ={}&\langle y,\int_{0}^{\infty}h(-u)T(u)(s\ide-A)^{-1}x\ du\rangle_{X',X},
		\end{align*}
		since the integral exists strongly. Because $(s\ide-A)^{-1}$ maps $X$ onto $\DA$, this completes the proof.
		\end{proof}
		\begin{remark}[to Lemma \ref{grational}]
		\begin{enumerate}
		\item
		The proof shows that the Lemma is still valid if we assume more generally that $g$ is the Fourier-Laplace transform of a Borel measure on $(-\infty,0]$ with bounded variation, i.e.,
		\begin{equation*}
		g(is)=\int_{\R}e^{-ist} \ \mu(dt).
		\end{equation*}
		Then the operator $g(A)$ reads
		\begin{equation*}
		  g(A)x=\int_{0}^{\infty}T(s)x \ \mu(ds),
		 \end{equation*}
		 which is the well-known \textit{Phillips-calculus}.
		 \item Also, more generally, $g$ can be assumed to be the Fourier transform of 
		 an $h\in L^{p}(\R)$ with support in $(-\infty,0]$ and $1\leq p\leq 2$. Since $T(.)$ is exponentially stable, 
		 $f=\langle y,T(.)x\rangle\in L^{q}(\R)$ for all $q\geq1$. Thus, $h\ast f_{ext}$ is still in $\Ltwo$ by Young's inequality 
		 (for which you choose $q$ such that $\frac{1}{p}+\frac{1}{q}=1+\frac{1}{2}$). The rest of the proof stays the same.
	
		 \end{enumerate}
		\end{remark}

		We collect some basic results of our calculus.
		\begin{theorem}\label{2mainthm}The functional calculus has the following properties:
		\begin{enumerate}[i)]
		\item \label{mainthm2:0}Define $\HB=\left\{g\in\Hinf:\gL\in\Bo(X)\right\}$. Then,
		\begin{equation*}
		\Phi:\HB\rightarrow\Bo(X),\quad g\mapsto\gL
		\end{equation*}
		is an algebra homomorphism.
		\item \label{mainthm2:1}If $P\in\Bo(X)$ commutes with $A$, $PA\subset AP$, i.e.,
		\begin{equation}\label{eq:APPA}
		\DA\subset D(AP) \quad\text{ and }\quad PAx_{1}=APx_{1} \quad \forall x\in\DA,
		\end{equation}
		then $P$ commutes with $g_{\Lambda}(A)$ for any $g\in\Hinf$. 
		In particular, $T(t)$ commutes with $\gL$ for any $t>0$.
		\item \label{mainthm2:2}For $g_{\mu}(z)=\frac{1}{\mu-z}$ we have $g_{\mu,\Lambda}(A)=R(\mu,A)$ for all $\mu$ with $\Re(\mu)>0$.
		\item \label{mainthm2:3}For $g_{t}(s)=e^{ts}$ we have $g_{t,\Lambda}(A)=T(t)$ for all $t\geq0$.
		\end{enumerate}
		\end{theorem} 
		\begin{proof}
		\ref{mainthm2:0}) Let $g_{1},g_{2}$ be in $\HB$. By Theorem \ref{mainthm} \ref{mainthm3}), $\gLot$ is an extension of $\gLo\gLt$. Since the latter is a bounded operator defined on $X$, also $\gLot\in\Bo(X)$. Thus, $\gLot\in\HB$. The rest is clear from Theorem \ref{mainthm}.\newline
		\ref{mainthm2:1}) 
		Using the Laplace transform, it is easy to see that (\ref{eq:APPA}) implies that $P$ commutes with $T(t)$ for any $t\geq0$. In fact, 
		(\ref{eq:APPA}) implies
		\begin{equation*}
		P(s\ide-A)x=(s\ide-A)Px\qquad \forall x\in\DA.
		\end{equation*} 
		For $s\in\rho(A)$ this yields
		\begin{equation*}
		(s\ide-A)^{-1}Px=P(s\ide-A)^{-1}x\qquad\forall x\in X.
		\end{equation*}
		This is nothing else than
		\begin{equation}\label{proofeqtmp}
		\int_{0}^{\infty}e^{-st}T(t)Px \ dt=P\int_{0}^{\infty}e^{-st}T(t)x\ dt\qquad\forall x\in X.
		\end{equation}
		Since $P$ is bounded, 
		\begin{equation*}
		P\int_{0}^{\infty}e^{-st}T(t)x=\int_{0}^{\infty}e^{-st}PT(t)x \ dt,
		\end{equation*}
		therefore, by (\ref{proofeqtmp}), we deduce
		\begin{equation*}
		T(t)Px_{0}=PT(t)x_{0} \qquad\forall x_{0}\in X,
		\end{equation*}
		 since the Laplace transform is injective.
		Let $y\in X'$ and $x\in\DA$. Similar as in the proof of Theorem \ref{thm:g(A)}, we deduce
		\begin{align*}
		\langle y,g(A)T(t)Px_{1}\rangle={}&(\Ds_{g,y}Px_{1})(t)\\
							={}&[M_{g}\big(\langle y,T(.)Px_{1}\rangle\big)](t)\\
							={}&[M_{g}\big(\langle y,PT(.)x_{1}\rangle\big)](t)\\
							={}&\langle P'y,g(A)T(t)x_{1}\rangle\\
							={}&\langle y,Pg(A)T(t)x_{1}\rangle.
		\end{align*}
		Hence, $Pg(A)x_{1}=g(A)Px_{1}$ for all $x_{1}\in\DA$. Now, let $x\in D(\gL)$. Since $P\in\Bo(X)$
		\begin{equation*}
		P\gL x=\lim_{\lambda\to\infty}\lambda Pg(A)\RlA x.
		\end{equation*}
		By  the already shown commutativity on $\DA$, the right hand side equals
		\begin{equation*}
		\lim_{\lambda\to\infty}\lambda g(A)P\RlA x.
		\end{equation*}
		Clearly, $P\RlA=\RlA P$, thus
		\begin{equation*}
		P\gL x=\lim_{\lambda\to\infty}\lambda g(A)\RlA Px.
		\end{equation*}
		Since the limit exists, $Px\in D(\gL)$ and $P\gL x=\gL Px$.
		\newline \ref{mainthm2:2}) This is an application of Lemma \ref{grational}. In fact, observe that
		\begin{equation*}
		g_{\mu}(i\omega)=\frac{1}{\mu-i\omega}=\mathcal{F}(e^{\mu s}|_{(-\infty,0)})(\omega).
		\end{equation*}
		Therefore, 
		\begin{equation*}
		g_{\mu}(A)x=\int_{0}^{\infty}e^{-\mu t}T(t)x \ dt=R(\mu,A)x.
		\end{equation*}
		\newline \ref{mainthm2:3}) Clearly, the function $g_{t}$ is in $\Hinf$. Let us recall the following property of the Fourier/Laplace transform. For $f\in\Ltwo$ we define $f_{ext}$ to be the extension of $f$ by $0$ to the whole real axis. Now we have that
		\begin{align*}
		e^{it\omega}\Lap(f)(i\omega)={}&e^{it\omega}\mathcal{F}(f_{ext})(\omega)\\
					={}&\mathcal{F}\big(f_{ext}(.+t)\big)(\omega)\\
				        ={}&\Lap\big(f(.+t)|_{(0,\infty)}\big)(i\omega)+\Lap\big(f(.+t)|_{(-t,0)}\big)(i\omega).
		\end{align*}
		Thus, 
		\begin{equation*}
		M_{g_{t}}f=\Lap^{-1}\Pi (g_{t}\cdot\Lap(f))=f(.+t)|_{(0,\infty)}.
		\end{equation*}
		Set $f=\langle y,T(.)x_{1}\rangle$ for $x_{1}\in\DA$ and $y\in X'$. By definition of $\Ds_{g,y}$ and Theorem \ref{thm:g(A)}, we have
		\begin{equation*}
		\langle y,g_{t}(A)T(u)x\rangle=(\Ds_{g_{t},y}x_{1})(u)=\langle y,T(u+t)x\rangle=\langle y,T(t)T(u)x\rangle
		\end{equation*}
		for all $x\in\DA$ and $y\in X'$. Therefore, $g_{t}(A)=T(t)|_{\DA}$ and the assertion follows. 
	        \end{proof}
	        We conclude this section by proving that the main identity of the construction, (\ref{eq:thm:g(A)X}),
	         can be extended for the Lambda extension.
	        \begin{proposition}
	        For $y\in X'$ and $g\in\Hinf$ we have that
	        \begin{equation}\label{eq:weissforlambdaext}
			\langle y,\gL T(.)x \rangle=[\Ds_{g,y}x](.)\qquad \forall x\in D(\gL).
		\end{equation}
		\end{proposition}
		\begin{proof}
			Let $x\in D(\gL)$. Since (by \cite[Lemma II.3.4]{engelnagel})
			\begin{equation*}
			x_{n}:=R(n,A)x\to x\qquad \text{as }n\to\infty,
			\end{equation*}
			we have that $\Ds_{g,y}x_{n}\to\Ds_{g,y}x$ in $\Ltwo$. 
			Furthermore, $x_{n}\in\DA$ and $g(A) x_{n}\to\gL x$ as $n\to\infty$ by the definition of the Lambda extension. 
			Now, the assertion follows from (\ref{eq:thm:g(A)X}) and since $T(t)$ commutes with $\gL$ by Theorem \ref{2mainthm}.\ref{mainthm2:1}.
		\end{proof}

                     \section{Sufficient conditions for a bounded calculus}\label{sec:boundedcalc}
                     \subsection{Exact Observability by Direction}
                      In order to give a sufficient condition for a bounded functional calculus, we introduce a refined notion of observability.
                      \begin{definition}
                      For an operator $C\in\Bo(X_{1},Y)$, the pair $(C,A)$ is called \textbf{exactly observable by direction} if there exist  $m,K>0$ such that for every $x\in\DA$ there is a $y_{x}\in Y'$ with $\|y_{x}\|_{Y'}=1$ such that
                      \begin{equation}
                      \label{eq:obsbd}
                         K\|x\|\leq\|\langle y_{x},CT(.)x\rangle_{Y',Y}  \|_{\Ltwo}\leq m\|x\|.
                      \end{equation}   
                      \end{definition}
                      \begin{theorem}\label{thmexbydir}
                      Let $C\in\Bo(X_{1},Y)$ be exactly observable by direction. Then, $g\mapsto \gL$ is a bounded $\Hinf$-calculus with
                      \begin{equation}
                      	\|\gL\|\leq\frac{m}{K}\|g\|_{\infty},
		   \end{equation}
		   where $m,K$ are the constants from $(\ref{eq:obsbd})$.
                      \end{theorem}
                      \begin{proof}
                      Let $x\in\DAt$. Then, there exists a $y_{x}\in X'$ with norm 1 such that
                      \begin{align*}
                      K\|g(A)x\|\leq{}& \|\langle y_{x},CT(.)g(A)x\rangle_{Y',Y}  \|_{\Ltwo}\\
                                        ={}& \|\langle y_{x},Cg(A)T(.)x\rangle_{Y',Y}  \|_{\Ltwo},
                      \end{align*}
                      where we used that $g(A)$ commutes with the semigroup. 
                      By the proof of Lemma \ref{le;compositewadm}, eq. (\ref{eq:lemCg(A)}) 
                      (there, we do not need weak admissibility of $C$) we obtain
                      \begin{equation*}
                       \|\langle y_{x},Cg(A)T(.)x\rangle_{Y',Y}  \|_{\Ltwo} = \|M_{g}\langle y_{x},CT(.)x\rangle_{Y',Y}  \|_{\Ltwo}.                
                       \end{equation*}
                        This, we can estimate by the norm of the Toeplitz operator, Lemma \ref{le:propMg}.\ref{boundofToeplitz}., 
                        and by using the assumption. Hence,
                        \begin{equation*}
                        \|M_{g}\langle y_{x},CT(.)x\rangle_{Y',Y}  \|_{\Ltwo}\leq m\|g\|_{\infty}\|x\|,
                        \end{equation*}
                       Altogether, we have for $x\in\DAt$
                       \begin{equation}
                       \|g(A)x\|\leq\frac{m}{K}\|g\|_{\infty}\|x\|,
                       \end{equation}
                       which proves the assertion, since $\DAt$ is dense.
                       \end{proof}

 \subsection{Exact Observability vs. Exact Observability by Direction}\label{sec:obsvobbd} 
 
                       In this section we are going to investigate the relation between our `weak' calculus and the `strong' approach for separable Hilbert spaces done in \cite{ZwartAdmissible}. For this, $X$ will be a separable Hilbert space for the rest of the section. To be consistent with our notation of the duality brackets, the inner product of $X$ is linear in the second component.
                       The strong calculus  is constructed by choosing the output mapping
                       \begin{equation*}
                       \Ds_{g}:X\rightarrow\LX:\qquad x\mapsto M_{g}(T(.)x).
                       \end{equation*}
                       Note that $M_{g}=\Lap^{-1}\Pi_{X}(g\cdot\Lap)$ is now defined via the Laplace transform on $\LX$ and the projection $\Pi_{X}$. Then, $g_{s}(A)\in\Bo(X_{1},X)$ is the admissible operator from Lemma \ref{weiss} such that
                       \begin{equation}\label{outputstrong}
                       \Ds_{g}(.)x=g_{s}(A)T(.)x,
                       \end{equation}
                       for $x\in\DA$. 
                       To provide a sufficient condition that $g_{s}(A)$ is bounded the following notion is used in  \cite{ZwartAdmissible}.
                         \begin{definition}Let $Y$ be a separable Hilbert space.
                         For an operator  $C\in\Bo(X_{1},Y)$, the pair $(C,A)$ is called \textbf{exactly observable} 
                         if there exist $m,K>0$ satisfying
                         \begin{equation}
                         \label{eq:exobs}
                         K\|x\|\leq\|CT(.)x\|_{\LY}\leq m\|x\|
                         \end{equation}
                         for all $x\in\DA$.
                         \end{definition}      
                         \begin{remark} \normalfont  
                         \label{rem:exobsobsby}                      
                          \begin{enumerate}
                          \item
                            Since we have that for $\|y\|_{Y}=1$ and $x\in\DA$ 
                          \begin{equation*}
                             |\langle y,CT(t)x\rangle_{Y}  |\leq\|CT(t)x\|
                           \end{equation*}
                            holds. This gives a correspondence between exact observability (\ref{eq:exobs}), 
                            and exact observability by direction (\ref{eq:obsbd}), as indicated in the following scheme
                            \begin{equation*}
                          \begin{matrix}
                             K\|x\|&\leq&\|\langle y_{x},CT(.)x\rangle\|_{\LY}&\leq &m\|x\|&\text{(Ex.Obs. by dir.)}\\
                             &\Downarrow&&\Uparrow&&\\
                            K\|x\|&\leq&\|CT(.)x\|_{\LY}&\leq &m\|x\|&\text{(Ex.Obs.)}                            
                          \end{matrix}
                          \end{equation*}
                           \item
                          We remark that, provided $\exists m>0 $,
                          \begin{equation*}
                          \forall x\in\DA\exists y_{x},\|y_{x}\|=1,\qquad \langle y_{x},CT(.)x\rangle\|_{\LY}\leq m\|x\|,
                          \end{equation*}
                           $(C,A)$ is not exactly observable by direction iff there exists a sequence $\left\{x_{k}\right\}\subset \DA$ with $\|x_{k}\|=1$, $k\in\N$ such that
                           \begin{equation*}
                           \|\langle y,CT(t)x_{k}\rangle\|_{\Ltwo}<\frac{1}{k}
                           \end{equation*}
                           for all $y\in Y$ with $\|y\|_{Y}=1$. \newline
                           We will use this characterization later.
                           \item Exact observability can be defined for general Banach spaces $X$, $Y$ since the definition does not need the Hilbert space structure.
                           \end{enumerate}
                           \end{remark}
                           The following result is the Hilbert space counterpart of Theorem \ref{thmexbydir} for the strong calculus, see \cite{ZwartAdmissible}. 
	                  \begin{theorem}\label{thm:strongcalcbdd}
                           If there exists an operator $C\in\Bo(X_{1},Y)$ such that $(C,A)$ is exactly observable, then $g_{s}(A)$ is bounded. \newline
                           Hence, the strong calculus, $g\mapsto (g_{s}(A))_{\Lambda}$ 
                           (where $(g_{s}(A))_{\Lambda}$ denotes the Lambda extension of $g(A)$)  is bounded. 
                           \end{theorem}
                           Moreover, the notions of weak and strong calculus coincide (for a separable Hilbert space $X$).
                           To prove this, we make use of the following elementary result.
                           \begin{lemma}
                           \label{le:LapPi}
                           Let $X$ be a separable Hilbert space, 
                          $f\in\LX$, $g\in\Ht(X)$ and $h\in L^{2}(i\R,X)$. Then, for $y\in X$
                           \begin{enumerate}[i.]
                           \item $\langle y,\Lap f\rangle=\Lap\langle y,f\rangle$ and $\langle y,\Lap^{-1}g\rangle=\Lap^{-1}\langle y,g\rangle$, 
                           \item $\langle y,\Pi_{X}h\rangle=\Pi\langle y,h\rangle$.
                          \end{enumerate}
                           \end{lemma}
                           \begin{proof}
                           The first assertion holds because $\Lap f$ and $\Lap^{-1}g$ exist strongly and by the continuity of the inner product. To see the second assertion, we use that $L^{2}(i\R,X)=\Ht(X)\oplus \Hto(X)$. Hence, we can find $h_{1}\in\Ht(X)$ and $h_{2}\in\Hto(X)$ such that $h=h_{1}+h_{2}$. 
                           From the first part of this lemma we have that $\langle y,h_{1}\rangle\in\Ht$ and   $\langle y,h_{2}\rangle\in\Hto$ which yields 
                           \begin{equation*}
                           \langle y, \Pi_{X}h\rangle=\langle y,h_{1}\rangle=\Pi\langle y,h_{1}\rangle=\Pi\langle y,h\rangle.
                           \end{equation*}
                           \end{proof}
                           \begin{theorem}
                           \label{coincidencestrongweak}
                           Let $X$ be a separable Hilbert space. Then $g(A)=g_{s}(A)$ and therefore, $\gL=(g_{s}(A))_{\Lambda}$.
                           \end{theorem}
                           \begin{proof}
                          It suffices to show that 
                           \begin{equation}
                           \label{proof:strongweakcalculus}
                             \langle y, g(A)T(t)x\rangle=\langle y,g_{s}(A)T(t)x\rangle
                           \end{equation}
                           for $t>0,y\in X'$ and $x\in\DA$. By Theorem \ref{thm:g(A)} and its counterpart for the strong calculus (see (\ref{outputstrong})), we have that
                           \begin{align*}
                           \langle y, g(A)T(.)x\rangle={}&\Ds_{g,y}x,\\
                           \langle y,g_{s}(A)T(.)x\rangle={}&\langle y,\Ds_{g} x\rangle.
                           \end{align*}  
                           where $\Ds_{g}x=M_{g}(T(.)x)$ with $M_{g}\in\Bo(\LX)$. By the definition of $M_{g}$ and Lemma \ref{le:LapPi} we see that
                           \begin{align*}
                           \langle y,\Ds_{g}x\rangle={}&\langle y,\Lap^{-1}\Pi_{X}(g\cdot \Lap [T(.)x)]\rangle\\
                           					   ={}&\Lap^{-1}\Pi(g\cdot \Lap[\langle y,T(.)x\rangle])\\
					     		   ={}&M_{g}(\langle y,T(.)x\rangle)\\
					   		   ={}&\Ds_{g,y}x,
	                \end{align*}
	                where this last $M_{g}$ is an element in $\Bo(\Ltwo)$.
	                Hence, the equality in (\ref{proof:strongweakcalculus}) holds for almost every $t>0$.  
	                Since both functions are continuous in $t$, it holds even point-wise and in particular for $t=0$.
                           \end{proof} 
                            
                           \begin{remark}
                           A consequence of Theorem \ref{coincidencestrongweak} is that  the weak calculus of Section 2 is automatically  admissible in the separable Hilbert space case.                    
                            \end{remark}
                           \begin{proposition}For finite dimensional $Y$ and weakly admissible $C\in\Bo(X_{1},Y)$, exact observability and exact observability by direction of $(C,A)$ are equivalent.
                           \end{proposition}
                           \begin{proof}
                           Since for finite dimensional $Y$ the notions of admissibility and weak admissibility coincide, in the view of Remark \ref{rem:exobsobsby} it remains to show that 
                           (\ref{eq:exobs}) implies (\ref{eq:obsbd}). Assume that $(C,A)$ is not exactly observable by direction. 
                           Hence, there exists a sequence $x_{n}$ in $\DA$ with $\|x_{n}\|=1$ such that
                           \begin{equation}\label{prop:finiteNEG}
                           \|\langle y,CT(.)x_{n}\rangle_{Y}  \|_{\Ltwo}<\frac{1}{n} \qquad \forall y\in Y', \|y\|_{Y}=1,
                           \end{equation}
                           for all $n\in\N$. 
                           Let $\left\{\phi_{k}:k=1,..,N\right\}$ and $\left\{\phi'_{k}:k=1,..,N\right\}$ be 
                           bases of $Y$ and $Y'$ respectively, such that $\|\phi_{k}\|=1$ and $\langle \phi'_{k},\phi_{j}\rangle=\delta_{kj}$ for $k,j=1,É,N$. 
                           Then, for $t\geq0$
                           \begin{align*}
                           CT(t)x_{n}={}&\sum_{k=1}^{N}\langle \phi'_{k},CT(t)x_{n}\rangle\phi_{k}\\
                           \Rightarrow  \|CT(t)x_{n}\|_{Y}^{2}\leq{}&N\sum_{k=1}^{N}|\langle \phi'_{k},CT(t)x_{n}\rangle|^{2}.
                           \end{align*}
                           Integrating and using (\ref{prop:finiteNEG}) this yields
                               \begin{equation*}
                            \|CT(.)x_{n}\|_{\LY}\to0
                            \end{equation*}
                            for $n\to\infty$. This contradicts the exact observability of $(C,A)$.
                            \end{proof}
                           Finally, we give an example that, even given admissibility, in general exact observability does not imply observability by direction,
                           \begin{example}\normalfont\label{examplefinal}
                           We consider a Hilbert space $X$ with orthonormal basis $\left\{\phi_{n}\right\}_{n\in\N}$ and a set $\left\{\lambda_{n},n\in\N\right\}\subset\R_{-}$. Define an exponentially stable semigroup $T$ by
                           \begin{equation*}
                           T(t)\sum_{n=1}^{\infty}x_{n}\phi_{n}=\sum_{n=1}^{\infty}e^{\lambda_{n}t}x_{n}\phi_{n},\qquad t\geq0.
                           \end{equation*}
                           It can be shown that the generator of $T$ is given by
                           \begin{equation*}
                                        Ax=\sum_{n=1}^{\infty}\lambda_{n}x_{n}\phi_{n},
                           \end{equation*}                  
                           with $\DA=\left\{x\in X:\sum_{n=1}^{\infty} |\lambda_{n}x_{n}|^{2}<\infty\right\}$. As the observer $C$, we take the square root of $(-A)$, which is given by
                           \begin{equation*}
                           C\sum_{n=1}^{N}x_{n}\phi_{n}=\sum_{n=1}^{N}\sqrt{-\lambda_{n}}x_{n}\phi_{n},
                           \end{equation*}
                           and  domain $\DA=\left\{x\in X:\sum_{n=1}^{\infty} |\sqrt{\lambda_{n}}x_{n}|^{2}<\infty\right\}$ . \newline Define $f_{n}(.)=\sqrt{-2\lambda_{n}}e^{\lambda_{n}.}$.                           From \cite{nikolski2} Theorem D.4.2.2. (and the appropriate version for the left half-plane) we know that $f_n$ is a Riesz sequence in $\Ltwo$ if and only if there exists a $\rho>0$ such that
                           \begin{equation} \label{carlesoncond}
                           \prod_{m=1,m\neq n}^{\infty}\left|\frac{\lambda_{n}-\lambda_{m}}{\lambda_{n}+\bar{\lambda}_{m}}\right|\geq\rho \qquad \forall n\in\N.
                           \end{equation}
                           Consider now a sequence $\lambda_{n}$ which satisfies (\ref{carlesoncond}). A possible choice is $\lambda_{n}=-2^n$ (see \cite{garnett} page 278). Since $f_n$ is a Riesz sequence, there exist constants $m,M>0$ such that
                           \begin{equation}\label{rieszbasis}
                           m\sum_{n=1}^{N}|c_{n}|^{2}\leq\left\|\sum_{n=1}^{N}c_{n}f_{n}\right\|_{\Ltwo}^{2}\leq M\sum_{n=1}^{N}|c_{n}|^{2},
                           \end{equation}
                           for all finite sequences of complex numbers $(c_{1},...,c_{N})$. Let us apply these results to our situation. Define 
                           \begin{equation*}
                           x_{N}=\sum_{n=1}^{N}\frac{1}{\sqrt{N}}\phi_{n}.
                           \end{equation*}
 		       Then, $\|x_{N}\|=1$ and for all $y\in X$ with $\|y\|^{2}=\sum_{n=1}^{\infty}|y_{n}|^{2}=1$ there holds
		       \begin{align*}
		        \|\langle y,CT(.)x_{N}\rangle  \|_{\Ltwo}^{2}={}&\|\sum_{n=1}^{N}\sqrt{-\lambda_{n}}e^{\lambda_{n}.}x_{N,n}y_{n}\|_{\Ltwo}^{2}\\
		        ={}&\frac{1}{2}\|\sum_{n=1}^{N}x_{N,n}y_{n}f_{n}\|_{\Ltwo}^{2}\\
		        \leq{}&\frac{M}{2}\sum_{n=1}^{N}\frac{1}{N}|y_{n}|^{2}\\
		        \leq{}&\frac{M}{2N},
		        \end{align*}
		        where we used (\ref{rieszbasis}). Hence, $(C,A)$ is not exactly observable by direction (see Remark \ref{rem:exobsobsby}).\newline
		        However, by
		        \begin{align*}
		         \|CT(t)x\|_{\LX}^{2}={}&\frac{1}{2}\int_{0}^{\infty}\|\sum_{n=1}^{N}x_{n}f_{n}(t)\phi_{n}\|^{2}dt\\
		         				     ={}&\frac{1}{2}\int_{0}^{\infty}\sum_{n=1}^{N}|x_{n}|^{2}|f_{n}(t)|^{2}dt=\frac{1}{2}\|x\|^{2},
		         \end{align*}
		         we see that $(C,A)$ is exactly observable and, therefore, by Theorem \ref{thm:strongcalcbdd}, we obtain a bounded functional calculus.
                           \end{example}
                           \begin{remark}
                           Let us consider the situation of Example \ref{examplefinal}, but with $\lambda_{n}=\lambda_{0}<0$ for all $n$. Then, for $x\in\DA$
                           \begin{equation*}
                            \|\langle y,CT(.)x\rangle  \|_{\Ltwo}^{2}=\int_{0}^{\infty}|\langle y,\sqrt{-\lambda_{0}}e^{\lambda_{0}t}x\rangle|^{2} \ dt=\frac{1}{2}|\langle y,x\rangle|^{2}.
                           \end{equation*}
                           If we choose $y_{x}=\frac{x}{\|x\|}$, we get
                           \begin{equation*}
                            \|\langle y_{x},CT(.)x\rangle  \|_{\Ltwo}^{2}=\frac{1}{2}\|x\|^{2},
                            \end{equation*}
                            hence, $(C,A)$ is exact observable by direction. \newline
                            We remark that, independent of the choice of  the sequence $\left\{\lambda_{n}\right\}$, the solution to the Lyapunov equation
                            \begin{equation*}
                            AP+PA'=-C'C
                            \end{equation*}
                            is $P=\sqrt{2}\ide$.
                           
                           \end{remark}
                           \section{Further results}\label{sec:furtherresults}
                           \subsection{Relation to the natural $\mathcal{H}^{\infty}$-calculus}\label{subsection:Relationnaturalcalc}
                           After developing our calculus and giving sufficient conditions when it is bounded, we want to make some remarks about its consequences. 
                           As stated in the beginning, the uniqueness of a functional calculus is not clear at all, 
                           see \cite[Sections 2.8, 5.3 and 5.7.]{haasesectorial} 
                           In our case, the `straight-forward' question is to understand the relation 
                           with the natural (sectorial/halfplane) calculus, see Section \ref{naturalcalc}. \newline                
                           Since $A$ generates an exponentially stable semigroup, 
                           $A$ is a (strong) half-plane operator of type $\omega<0$, 
                           which means that
                           \begin{equation*}
                            \sigma(A)\subset\overline{L_{\omega}},
                            \qquad \sup_{\lambda\in R_{\omega'}}\|(\omega'-\Re(\lambda))\RlA\|<\infty, \quad \forall \omega'>\omega,
                            \end{equation*}
                            where $\overline{L_{\omega}}=\left\{z\in\C:\Re(z)\leq\omega\right\}$ is a left- and 
                            $R_{\omega'}=\C\setminus\overline{L_{\omega'}}$ is a right-half-plane. 
                            In \cite{battyhaasemubeen,haasehalfplaneoperators,mubeenPhD} it is shown 
                            that for such operators a \textit{natural} calculus can be defined in a similar manner 
                            as for sectorial operators. 
                            In the spirit of the a general \textit{meromorphic calculus} defined in \cite{haasesectorial}, 
                            one constructs a \textit{primary calculus}, 
                            which consists of a homomorphism mapping a subalgebra of functions to $\Bo(X)$, first.
                            Then, this is extended algebraically to $\Hinf$ by using \textit{regularizers}, see \cite[Chapter 1]{haasesectorial}. 
                            \begin{proposition}[Prop. 2.3. in \cite{battyhaasemubeen}, Prop. 3.2.6. in \cite{mubeenPhD}]
                            For 
                            \begin{equation*}
                            \Hinfone:=\left\{f\in\Hinf:|f(z)|=\mathcal{O}(|z|^{-\alpha})\text{ as }|z|\to\infty,\text{ with }\alpha>1\right\},
                            \end{equation*}
                            the mapping $\Hinfone\rightarrow\Bo(X), f\mapsto f_{HP}(A)$, where
                            \begin{equation*}
                            f_{HP}(A):=\int_{i\R-\epsilon}f(z)R(z,A)\ dz, \qquad \omega<-\epsilon<0.
                            \end{equation*}
                            is a homomorphism (the integration goes from $-i\infty-\epsilon$ to $i\infty-\epsilon$ 
                            and does not depend on $\epsilon$).
                            \end{proposition}
                            For $f\in\Hinf$, $\frac{f(z)}{(1-z)^{2}}\in\Hinfone$. 
                            Since the function $e(z)=(1-z)^{-2}\in\Hinfone$ yields $e_{HP}(A)=R(1,A)$, which is injective, one defines 
                            \begin{equation*}
                            f_{HP}(A)=(\ide-A)(e\cdot f)_{HP}(A)
                            \end{equation*}
                            (this is independent of the regularizer $e$) and, 
                            furthermore we have a similar result as for the sectorial calculus (see \cite[Lemma 3.3.1]{haasesectorial}).
                            \begin{lemma}\label{lem:halfplanelaplace}
                            For $g\in\Hinfone$ and  $\epsilon'>0$ 
                            there exists a unique $h\in L^{1}(\R_{-})$ such that $g(z)=\Lap(h)(z)$ for all $\Re(z)\leq-\epsilon'$. 
                            Moreover,
                            \begin{equation*}
                            	h(t)=\frac{1}{2\pi i}\int_{i\R-\epsilon} g(z)e^{zt} \ dz, \qquad -\epsilon'<-\epsilon<0.
			   \end{equation*}
			   By properties of $\Lap$, $g(z-\epsilon')=\Lap(h(.)e^{.\epsilon'})(z)$ for $z\in\C_{-}$.
			   \end{lemma}
			   \begin{proof}
			   Using Cauchy's formula we get for $\Re(z_{0})\leq-\epsilon'$
			   \begin{align*}
			   	g(z_{0})={}&\frac{1}{2\pi i}\int_{i\R-\epsilon}\frac{g(z)}{z-z_{0}}\ dz=
						\frac{1}{2\pi i}\int_{i\R-\epsilon}g(z)\int_{0}^{\infty}e^{(z_{0}-z)t}\ dt\ dz\\
					    ={}&\int_{0}^{\infty}e^{z_{0}t}\int_{i\R-\epsilon}g(z)e^{-zt}\ dz\ dt,
			 \end{align*}
			 where we used Fubini's theorem, which can be applied because 
			 \begin{equation*}
			 \int_{i\R-\epsilon}|g(z)|\int_{0}^{\infty}|e^{(z_{0}-z)t}|\ dt\ dz\leq\frac{1}{\epsilon'-\epsilon}\int_{i\R-\epsilon}|g(z)|\ dz<\infty,
			 \end{equation*}
			 where the integrability of $|g|$ follows since $g\in\Hinfone$.
			   \end{proof}
			   Next, we prove that $f_{\Lambda}(.-\epsilon)(A+\epsilon\ide)(A)$ equals indeed $f_{\Lambda}(A)$, as one would expect.
			  \begin{lemma}[Rescaling of weakly admissible calculus]\label{lem:rescaling}
			   For $f\in\Hinf$ and $\epsilon>0$ such that the rescaled semigroup $\tilde{T}(t)=e^{\epsilon t}T(t)$ with 
			   generator $A+\epsilon\ide$ is exponentially stable, it holds that 
			   \begin{equation*}
			   g_{\Lambda}(A+\epsilon\ide)=f_{\Lambda}(A),\qquad\text{where }g(z)=f(z-\epsilon)\quad\forall z\in\C_{-}.
			   \end{equation*}
			   
			   \end{lemma}
			    \begin{proof}
			   The proof relies on the fact that the projection $\Pi$ is translation-invariant. In fact, for $h\in L^{2}(\R)$,
			   \begin{align*}
			   	\Big(\Pi\mathcal{F}(h)\Big)(.-\epsilon)={}&\mathcal{F}(h|_{(0,\infty)})(.-\epsilon)=\mathcal{F}\Big(e^{i\epsilon x}\cdot(h(x)|_{(0,\infty)})\Big)(.)\\
					={}&\mathcal{F}\Big((e^{i\epsilon x}\cdot h(x))|_{(0,\infty)}\Big)(.)=\Pi\Big(\mathcal{F}(e^{i\epsilon x}\cdot h(x))\Big)(.)\\
					={}&\Pi\Big(\mathcal{F}(h)(.-\epsilon)\Big)(.)
			\end{align*}
			Using this and $\Lap(h)(.-\epsilon)=\Lap\Big(e^{\epsilon .}h(.)\Big)$, we see that for $x\in X,y\in X',t\geq0$ 
			\begin{align*}
			\Big[M_{g}\langle y,\tilde{T}(.)x\rangle\Big](t)={}&\Lap^{-1}\Pi\Big(\langle y,f(i.-\epsilon)\cdot\Lap(T(.)x)(i.-\epsilon)\rangle\Big)(t)\\
				={}&\Lap^{-1}\Big[\Pi\Big(\langle y,f(i.)\cdot\Lap(T(.)x)(i.)\rangle\Big)(.-\epsilon)\Big](t)\\
				={}&e^{\epsilon t}\Big[M_{f}\langle y,T(.)x\rangle\Big](t).
			\end{align*}
			By (\ref{eq:thm:g(A)X}) and letting $t\to0^{+}$ this yields the assertion.
			\end{proof}
			The following elementary result will be needed in the proof of the upcoming theorem.
			   \begin{lemma}\label{lem:LambdaExtprop}
			   Let $A$ generate a semigroup $T(t)$ and let $B$ be a closed operator such that 
			   $D(A)\subset D(B)$ and $B\RlA x=\RlA Bx$ on $D(A)$ for some $\lambda\in\rho(A)$. 
			   Let $x\in X$ such that $B\RlA x\in\DA$. Then, 
			   \begin{equation*}
			   x\in D(B)\quad\text{ and }\quad Bx=(\lambda-A)B\RlA x.
			   \end{equation*}
			   \end{lemma}
                            \begin{proof}
                            Define $A_{n}=n(T(1/n)-\ide)\in\Bo(X)$ for $n\in\N$. Then,
                            \begin{equation*}
                            z_{n}=A_{n}\RlA x\to A \RlA x,\qquad\text{ as }n\to\infty.
                            \end{equation*}
                            Since $B$ is closed and commutes with some resolvent, by taking the Laplace transform, it follows that $T(t)By=BT(t)y$ for $y\in\DA$ and $t\geq0$. Thus, since $z_{n}\in\DA$ and $B\RlA x\in\DA$,
                            \begin{equation*}
                            Bz_{n}=A_{n}B\RlA x\to A B\RlA x,\qquad\text{ as }n\to\infty.
                            \end{equation*}
                            By closedness of $B$,  $A\RlA x\in D(B)$ and $BA\RlA x=AB \RlA$. 
                            From $A\RlA x=\lambda\RlA x-x$ and because $\RlA x\in D(B)$, the assertion follows.
                             \end{proof}
			  			   Now, we are able to compare our weakly admissible calculus with the calculus for half-plane operators.
			   \begin{theorem}[Coincidence of Calculi]\label{thm:coincidenceofcalculi}
			   Let $A$ generate an exponentially stable semigroup $T(t)$. For all $f\in\Hinf$,
			   \begin{equation*}
			   f_{\Lambda}(A)=f_{HP}(A),
			   \end{equation*}
			   where $f_{\Lambda}(A)$ is the weakly admissible calculus and 
			   $f_{HP}(A)$ denotes the natural calculus for half-plane-operators from \cite{battyhaasemubeen,mubeenPhD}.
			   
			   \end{theorem}
			   \begin{proof}
			   First assume that $f\in\Hinfone$. 
			   Choose $\epsilon,\epsilon'$ such that $\omega<-\epsilon'<-\epsilon<0$, where $\omega$ is the growth bound of $T(t)$ and let $x\in X$.
			   By definition of the half-plane calculus and basic semigroup theory,
			   \begin{align*}
			    f_{HP}(A)x:={}&\int_{i\R-\epsilon}f(z)R(z,A)x\ dz=\int_{i\R-\epsilon}f(z)\int_{0}^{\infty}e^{-zt}T(t)x\ dt\ dz\\
			    		     ={}&\int_{0}^{\infty}\int_{i\R-\epsilon}f(z)e^{-zt}\ dz\ T(t)x\ dt=\int_{0}^{\infty}h(-t)T(t)x\ dt,
			    \end{align*}
			    where we used Fubini's theorem and Lemma \ref{lem:halfplanelaplace}. 
			    Let $\tilde{h}(.)=h(.)e^{.\epsilon'}$ and $\tilde{T}(t)=e^{\epsilon't}T(t)$, which is exponentially bounded.
			    Since $f(z-\epsilon')=\Lap(\tilde{h})(z)$ for $z\in\C_{-}$ and $\tilde{h}$ is in $L^{1}(\R_{-})$, 
			    we can apply Lemma \ref{grational} to $\tilde{T}$,
			    \begin{equation*}
			    \int_{0}^{\infty}h(-t)T(t)x\ dt=\int_{0}^{\infty}\tilde{h}(-t)\tilde{T}(t)x\ dt=f(.-\epsilon')_{\Lambda}(A+\epsilon'\ide)x.
			    \end{equation*}
			    By Lemma \ref{lem:rescaling}, we know that $f(.-\epsilon')_{\Lambda}(A+\epsilon'\ide)=f_{\Lambda}(A)$. 
			    Therefore,  
			    \begin{equation}\label{eq:coincalcH1}
			       f_{\Lambda}(A)=f_{HP}(A) \qquad \forall f\in\Hinfone.
			      \end{equation}
			    For general $f\in\Hinf$, by definition,
			    \begin{align*}
			    f_{HP}(A)={}&(\ide-A)^{2}(f\cdot e)_{HP}(A),\\D(f_{HP}(A))={}&\left\{x\in X:(f\cdot e)_{HP}(A)x\in \DAt\right\}
			    \end{align*}
			    where $e(z):=(1-z)^{-2}\in\Hinfone$. Since $(f\cdot e)\in\Hinfone$, by (\ref{eq:coincalcH1}), 
			    we have, for $x\in X$,
			    \begin{equation}\label{eq:coincalceq2}
			  (f\cdot e)_{HP}(A)x=(f\cdot e)_{\Lambda}(A)x=f_{\Lambda}(A)(\ide-A)^{-2}x,	
		           \end{equation}
			    where the last equality follows from Theorems \ref{2mainthm}.\ref{mainthm2:2}.\ , \ref{mainthm}.\ref{mainthm3}.\ .
			    If $x\in D(f_{\Lambda}(A))$, Lemma \ref{LebExt} implies
			     \begin{equation}\label{eq:coincalc2}
			    	f_{\Lambda}R(1,A)^{2}x=R(1,A)^{2}f_{\Lambda}(A)x,
			 \end{equation}
			 and thus, by (\ref{eq:coincalceq2}), $x\in D(f_{HP}(A))$ and $f_{\Lambda}(A)x=f_{HP}(A)x$. 
			 Thus, $f_{\Lambda}(A)\subset f_{HP}(A)$. 
			 If $x\in D(f_{HP}(A))$, we deduce that $f_{\Lambda}(A)R(1,A)^{2}x$ lies in $\DAt$  by(\ref{eq:coincalceq2}). 
			  Since $R(1,A)x\in D(f_{\Lambda}(A))$, using Lemma \ref{LebExt} again, 
			  we obtain  that $f_{\Lambda}(A)R(1,A)x\in\DA$.  
			  Thus, since $f_{\Lambda} $ is closed and commutes with the resolvents of $A$, 
			  Lemma \ref{lem:LambdaExtprop} yields $x\in D(f_{\Lambda}(A))$.
			   \end{proof}
			  \subsection{Discussion $\&$ Remarks}
			  Although we only consider exponentially stable semigroups in this work, Lemma \ref{lem:rescaling} shows how to define our calculus for more general $C_{0}-$semigroups.
			  \begin{definition}
			  Let $A$ generate a semigroup $T(.)$ with growth bound $\omega$ and let $v>\omega$. 
			  For $f\in\mathcal{H}(L_{v})$, where $L_{v}=\left\{z\in\C:\Re z<v\right\}$, we define
			  \begin{equation*}
			  f(A):=f(.+v)(A-v\ide)
			  \end{equation*}
			  where the right-hand-side is defined since $f(.+v)\in\Hinf$ and 
			  $A-v\ide$ generates the exponentially stable semigroup $e^{-vt}T(t)$.
			  \end{definition}
			  Theorem \ref{thm:coincidenceofcalculi} gives rise to some comments. 
			  First of all, we can immediately make use of the known results for the natural (half-plane)-calculus as for instance the
			  following important continuity result.
			  \begin{theorem}[Convergence Lemma, Prop.\ 3.3.5 in\cite{mubeenPhD} or Thm.\ 3.1 in \cite{battyhaasemubeen}]
			  \label{thm:convergencelemma}
			  Let $A$ generate a semigroup $T(.)$ with growth bound $\omega<0$. Let $v>\omega$ and 
			  $(f_{n})_{n\in\N}\subset\mathcal{H}(L_{v})$, where $L_{v}=\left\{z\in\C:\Re z<v\right\}$, such that 
			  \begin{equation}\label{weakstarconv}
			  f_{n}(t)\to f(t) \ \forall t\in L_{v}\text{ as }n\to\infty,\text{ and }\sup_{n\in\N}\|f_{n}\|_{\infty}<\infty.
			  \end{equation}
			  Then, $f\in\mathcal{H}(L_{v})$ and
			  \begin{equation*}
			  	f_{n}(A)x\to f(A)x\qquad\text{ for all }x\in\DA.
			\end{equation*}
			  \end{theorem}
			  
			 \begin{remark}[to Theorem \ref{thm:convergencelemma}]\label{rem:convergencelemma}
			 \begin{enumerate}
			 \item
			 The Convergence Lemma is not surprising in the view of the Toeplitz operator: 
			 By rescaling assume w.l.o.g that $v>0$. 
			 Therefore, the functions $f_{n}$ converge pointwise on $i\R$ and by Dominated Convergence one can see that
			 \begin{equation*}
			 M_{f_{n}}h\to M_{f}h\ \text{ in }\Ltwo\qquad \text{ as }n\to\infty. 
			 \end{equation*}
			 for $h\in\Ltwo$.
			 \item By \cite[Prop. F.4.]{haasesectorial}, 
			 for any $f\in\mathcal{H}(L_{v})$ there exists a sequence of rational functions $r_{n}$ from
			 \begin{equation*}
			 	\mathcal{R}^{\infty}(L_{v})=\left\{\frac{p}{q}:p,q\in\C[z],deg (p )\leq deg(q),\text{poles of }q\text{ are in } R_{v}\right\},
			\end{equation*}
			such that $r_{n}\to f$ pointwise on $L_{v}$ and $\|r_{n}\|_{\infty}\leq\|f\|_{\infty}$.
			 \end{enumerate}
			 \end{remark}
			 By Theorem \ref{thm:convergencelemma} and Remark \ref{rem:convergencelemma} 
			 we observe that the calculus is built of  approximations by simple operators. 
			 Thus, it often suffices to restrict on functions in $\mathcal{R}^{\infty}$ (for $v\leq0$, 
			 these are Laplace transforms of $L^{1}$ functions), to show a property of the calculus.\\
			 Although the calculi are the same, the construction of our weakly admissible calculus and 
			 the natural (half-plane) calculus is quite different. 
			 We want to emphasize that a meromorphic calculus, 
			 \cite[Chapter 1]{haasesectorial} is defined in a purely algebraic way. Namely, when
			 extending the primary calculus to more general functions by using regularizer. 
			 However, in our construction a crucial step was to take the Lambda extension of $f(A)$.\\
			 Finally, let us mention the following representation of the Toeplitz operator $M_{g}$,
			 \begin{equation*}
			 M_{g}h=\Big(\mathcal{F}^{-1}\big(g(i.)\cdot[\mathcal{F}h](.)\big)\Big)|_{(0,\infty)},\qquad h\in\Ltwo,g\in\Hinf,
			 \end{equation*}
			which has been used a couple of times in this work. 
			This shows the relation to Fourier multipliers, which occurred already in the study of the $\Hinf$ calculus, e.g. in 
			\cite{haasetransference}
			


\section*{Acknowledgments} 
We would like to thank Markus Haase and Jan Rozendaal for fruitful discussions.
The first author has been supported by the Netherlands Organisation for Scientific Research (NWO) within the project \textit{Semigroups with an Inner Function calculus}, grant no.\ 613.001.004.
    \small

\begin{thebibliography}{10}
\bibitem{battyhaasemubeen} Charles Batty, Markus Haase, and Junaid Mubeen. \newblock The holomorphic functional calculus approach to operator semigroups. \newblock Preprint, 2012.

\bibitem{cowlingdoustmcintoshyagi}
Michael Cowling, Ian Doust, Alan McIntosh, and Atsushi Yagi.
\newblock Banach space operators with a bounded {$H^\infty$} functional
  calculus.
\newblock {\em J. Austral. Math. Soc. Ser. A}, 60(1):51--89, 1996.

\bibitem{engelnagel}
Klaus-Jochen Engel and Rainer Nagel.
\newblock {\em One-parameter semigroups for linear evolution equations}, volume
  194 of {\em Graduate Texts in Mathematics}.
\newblock Springer-Verlag, New York, 2000.

\bibitem{garnett}
John~B. Garnett.
\newblock {\em Bounded {A}nalytic {F}unctions}, volume 236 of {\em Graduate
  Texts in Mathematics}.
\newblock Springer, New York, first edition, 2007.

\bibitem{haasesectorial}
Markus Haase.
\newblock {\em The {F}unctional {C}alculus for {S}ectorial {O}perators}, volume
  169 of {\em Operator Theory: Advances and Applications}.
\newblock Birkh\"auser Verlag, Basel, 2006.

\bibitem{haasehalfplaneoperators}
Markus Haase.
\newblock Semigroup theory via functional calculus.
\newblock Preprint, available at:
  \url{http://fa.its.tudelft.nl/ haase/files/semi.pdf}, 2006.

\bibitem{haasetransference}
Markus Haase.
\newblock{Transference principles for semigroups and a theorem of {P}eller}
\newblock{\em J. Funct. Anal.} 261(10):2959--2998, 2011.

\bibitem{JacobPartingtonPott}
Birgit Jacob, Jonathan~R. Partington, and Sandra Pott.
\newblock Admissible and weakly admissible observation operators for the right
  shift semigroup.
\newblock {\em Proc. Edinb. Math. Soc. (2)}, 45(2):353--362, 2002.

\bibitem{mcintoshHinf}
Alan McIntosh.
\newblock Operators which have an {$H_\infty$} functional calculus.
\newblock In {\em Miniconference on operator theory and partial differential
  equations ({N}orth {R}yde, 1986)}, volume~14 of {\em Proc. Centre Math. Anal.
  Austral. Nat. Univ.}, pages 210--231. Austral. Nat. Univ., Canberra, 1986.

\bibitem{mubeenPhD}
Junaid Mubeen.
\newblock {\em The bounded $\mathcal{H}^{\infty}$-calculus for sectorial,
  strip-type and half-plane operators}.
\newblock {PhD} thesis, University of Oxford, 2011.

\bibitem{nikolski2}
Nikolai~K. Nikolski.
\newblock {\em Operators, {F}unctions, and {S}ystems: {A}n {E}asy {R}eading.
  {V}ol. 2}, volume~93 of {\em Mathematical Surveys and Monographs}.
\newblock American Mathematical Society, Providence, RI, 2002.
\newblock Model operators and systems,.

\bibitem{vonneumann}
John von Neumann.
\newblock {\em Mathematical foundations of quantum mechanics}.
\newblock Princeton Landmarks in Mathematics. Princeton University Press,
  Princeton, NJ, 1996.

\bibitem{WeissAdmissible}
George Weiss.
\newblock Admissible observation operators for linear semigroups.
\newblock {\em Israel J. Math.}, 65(1):17--43, 1989.

\bibitem{yosida}
K{\^o}saku Yosida.
\newblock {\em Functional analysis}, volume 123 of {\em Grundlehren der
  Mathematischen Wissenschaften}.
\newblock Springer-Verlag, Berlin, sixth edition, 1980.

\bibitem{ZwartAdmissible}
Hans Zwart.
\newblock Toeplitz operators and $\mathcal{H}^{\infty}$-calculus.
\newblock {\em J. Funct. Anal.}, 263(1):167 -- 182, 2012.

\bibitem{ZwartJacobStaffans}
Hans Zwart, Birgit Jacob, and Olof Staffans.
\newblock Weak admissibility does not imply admissibility for analytic
  semigroups.
\newblock {\em Systems Control Lett.}, 48(3-4):341--350, 2003.

\end{thebibliography}

\end{document}